\documentclass[11pt,a4paper,twoside]{article}
\let\Hungarian\H
\usepackage[centertags]{amsmath}
\usepackage{amsfonts}
\usepackage{amssymb}
\usepackage{amsthm}
\usepackage{ulem}
\usepackage{epsfig}
\usepackage{subfigure}
\usepackage{color}
\usepackage{mathrsfs}
\usepackage[boldsans]{ccfonts}
\usepackage{array}
\usepackage{calc}
\usepackage{titlesec}
\usepackage{graphicx}
\usepackage{enumitem}
\usepackage{hyperref}

\newtheoremstyle{mystyle}%
{\topsep}% Space above
{\topsep}% Space below 
{\itshape\color{red}}% Body font
{}% Indent amount
{\bfseries\color{red}}% Theorem head font
{.}% Punctuation after theorem head
{.5em}% Space after theorem head
{}% Theorem head spec (can be left empty, meaning ‘normal’)
\newtheoremstyle{break}
  {\topsep}{\topsep}%
  {\itshape}{}%
  {\bfseries}{}%
  {\newline}{}%
\newtheoremstyle{break1}
  {\topsep}{\topsep}%
  {}{}%
  {\bfseries}{}%
  {\newline}{}%
\theoremstyle{break} \newtheorem{theorem}{Theorem}[section]
\theoremstyle{break} 
\theoremstyle{break} \newtheorem{remark}[theorem]{Remark}
\theoremstyle{break} \newtheorem{remarks}[theorem]{Remarks}
\theoremstyle{break} 
\theoremstyle{break1} \newtheorem*{definition}{Definition} 
\theoremstyle{break} \newtheorem{lemma}[theorem]{Lemma}
\theoremstyle{break} \newtheorem{corollary}[theorem]{Corollary}
\theoremstyle{break} 
\theoremstyle{break} 
\theoremstyle{break} \newtheorem{problem}[theorem]{Problem}
\theoremstyle{mystyle} \newtheorem{problem1*}{Problem}
\theoremstyle{break} 
\theoremstyle{break} \newtheorem{proposition}[theorem]{Proposition}
\theoremstyle{break} 
\numberwithin{equation}{section}

\newcommand{\hide}[1]{}
\newcommand{\N}{{\mathbb{N}}}
\newcommand{\R}{{\mathbb{R}}}
\newcommand{\D}{{\mathbb{D}}}
\newcommand{\C}{{\mathbb{C}}}

\renewcommand{\H}{{\mathbb{H}}}

\def\Aut{\mathop{{\rm Aut}}}

\def\A{\mathcal{A}(\D)}

\def\si{\mathop{{\rm \simeq}}}

\makeatletter
\def\blfootnote{\xdef\@thefnmark{}\@footnotetext}
\makeatother

\begin{document}

\vspace*{-2cm}

{\begin{center}
{\Large \bf Strict Wick--type deformation quantization on Riemann surfaces:\\[2mm] Rigidity and Obstructions}
\end{center}}

\medskip
\renewcommand{\thefootnote}{\arabic{footnote}}
\begin{center}
{\large Daniela Kraus, Oliver Roth, Sebastian Schlei{\ss}inger and Stefan Waldmann}\\[1mm]
\today
\end{center}
\medskip

\begin{abstract}
  Let $X$ be a hyperbolic Riemann surface. We study a convergent Wick--type star product $\star_X$ on $X$  which is induced by the canonical convergent star product~$\star_{\D}$ on the unit disk $\D$  via Uniformization Theory. While by construction, the resulting Fr\'echet algebras $(\mathcal{A}(X),\star_X)$ are strongly isomorphic for conformally equivalent Riemann surfaces, our work exhibits additional severe topological obstructions. In particular, we show that the  Fr\'echet algebra $(\mathcal{A}(X),\star_X)$ degenerates if and only if the connectivity of~$X$ is at least~$3$, and  $(\mathcal{A}(X),\star_X)$ is noncommutative if and only if $X$ is simply connected. We also explicitly determine the algebra $\mathcal{A}_X$ and the star product $\star_X$ for the intermediate case of  doubly connected Riemann surfaces $X$. As a perhaps surprinsing result, we  deduce that two such Fr\'echet algebras are strongly isomorphic if and only if either both Riemann surfaces are conformally equivalent to an (not neccesarily the same)  annulus or both are conformally equivalent to a punctured disk.
  \end{abstract}

\section{Introduction}

 \blfootnote{2020 \textit{Mathematics Subject Classification.} Primary 30F45, 30F35, 53D55; Secondary  53A55}
        \blfootnote{\textit{Key words.} Riemann surfaces, Fuchsian groups, convergent star product}
For a complex manifold $\mathbb{M}$ we denote by $\mathcal{H}(\mathbb{M})$  the set of all
complex--valued holomorphic functions on $\mathbb{M}$. Equipped with  the standard compact--open topology,  $\mathcal{H}(\mathbb{M})$  becomes a Fr\'echet space. In this note, we are interested in
 the particular case $\mathbb{M}=\Omega$, where
$$ \Omega:=\hat{\C}^2 \setminus \left( \left\{ (z,w) \in \C^2 \, : \,
    zw=1\right\}
\cup \{ (0,\infty) \} \cup \{(\infty,0)\} \right) \, , $$
and $\hat{\C}:=\C \cup \{\infty\}$ denotes the Riemann sphere. Our interest in the
Fr\'echet space $\mathcal{H}(\Omega)$ and its function--theoretic properties comes from the insight that 
$\mathcal{H}(\Omega)$  arises naturally in problems
related to strict deformation
quantization of the open unit disk $\D=\{z \in \C \, : \, |z|<1\}$, see \cite{BDS,CGR} and in particular
\cite{BW,KRSW}. In fact, the main result\footnote{In
  \cite{KRSW}  the more general situation of the unit ball in $\C^n$ is  considered. We restrict ourselves to the case $n=1$ since in this
  paper we are
  primarly interested in star products on Riemann surfaces.} of \cite{KRSW} shows that the largest
Fr\'echet algebra $(\mathcal{A}(\D),\star_{\D})$ of all real--analytic functions on $\D$ 
 generated via phase space reduction
from the polynomials in~$z$ and $\overline{z}$
such that the canonical formal star product $\star_{\D}$ of Wick--type (see \cite{CGR}) is continuous
 can solely be described in terms of $\mathcal{H}(\Omega)$ and its natural compact--open topology:

\begin{theorem}[\cite{KRSW}] \label{thm:KRSW}
   $$ \A=\big\{f : \D \to \C \, | \,  f(z)=F(z,\overline{z}) \text{ for all } z \in
     \D \text{ for some } F \in \mathcal{H}(\Omega) \big\}\, . $$
Moreover, the algebra $(\A,\star_{\D})$  endowed with the topology inherited from the topology of locally uniform convergence in $\mathcal{H}(\Omega)$ is a  Fr\'echet algebra.
\end{theorem}

We note that for each $f \in \A$ there is a unique $F \in \mathcal{H}(\Omega)$ such that $f(z)=F(z,\overline{z})$ for all $z\in \D$. This is a consequence of the identity principle, see e.g.~\cite[p.~18]{Range}.

\medskip

The point of departure of this paper is the \textit{conformal invariance} of
the Fr\'echet algebra 
$(\A,\star_{\D})$, that is,  its invariance w.r.t.~precomposition with the
elements of the group  $\Aut(\D)$  of  conformal automorphisms (biholomorphic
selfmaps) of $\D$. For later reference, we  state this crucial property explicitly:

\begin{proposition}[Conformal invariance of the Fr\'echet algebra $(\A,\star_{\D})$] \label{prop:1}
  Let $f,g \in \A$ and $\varphi \in \Aut(\D)$. Then $f \circ \varphi \in A(\D)$ and $(f \circ \varphi) \star_{\D} (g \circ \varphi)= (f \star_{\D} g) \circ \varphi$.
  \end{proposition}

We point out that the star product $\star_{\D}$ depends on a complex deformation parameter $\hbar$, which plays the role of Planck's constant, and can take any value in the deformation domain $$\mathscr{D}:=\C \setminus \{0,-1,-1/2,-1/3, \ldots\} \, .$$ Moreover, $\star_{\D}$ depends holomorphically on $\hbar \in \mathscr{D}$, see \cite[Theorem 4.1]{KRSW}\footnote{In \cite{KRSW} the deformation domain is $\C \setminus \{0,-1/2,-1/4,-1/6,\ldots\}$ which is just a simple scaled version of~$\mathscr{D}$.}. In most of our results, we hold $\hbar \in \mathscr{D}$ fixed and simply write $\star_{\D}$. In exceptional cases, when  dependence on~$\hbar$ is essential, we use the notation $\star_{\D,\hbar}$ instead. 

\medskip
As suggested in \cite{KRSW}   the conformal invariance  of $(\A,\star_{\D})$ makes it possible to transfer the star product $\star_{\D}$ from the unit disk $\D$ to any other
   Riemann surface $X$ which is covered by the unit disk, that is, to any
   Riemann surface of hyperbolic type.  We emphasize that, with one exception, in the
   following we will merely use the conformal invariance of the star product,
   but not how it is explicitly defined. We therefore refrain for the moment from discussing any explicit form of the star product. The exception occurs for doubly connected Riemann surfaces. In this case we will make use of an explicit formula for  $\star_{\D}$ which has recently been obtained in \cite{HeinsMouchaRoth1}  and  \cite{SchmittSchoetz2022}. See the discussion in Section~\ref{sec:Doublyconnected} for details.
   
%   as well as the asymptotic expansion formula for $\star_{\D,\hbar}$ in \cite[Theorem ??.??]{HeinsMouchaRoth1}. 

    \begin{definition}[Convergent Wick--star product on hyperbolic Riemann surfaces]
      Let $X$ be a hyperbolic Riemann surface and  $\pi  : \D \to X$ a universal covering projection. Then
      $$ \mathcal{A}(X):=\left\{ f :   X \to \C \, | \, f \circ \pi \in
        \A \right\}\,  $$
      and for any $f,g \in \mathcal{A}(X)$ we define $f \star_X g \in \mathcal{A}(X)$ by 
      $$\left( f \star_X g \right) \circ \pi:= (f \circ \pi) \star_{\D} (g
      \circ \pi) \, .$$
        
      Moreover, we equip $\mathcal{A}(X)$ with the topology inherited from the topology of the Fr\'echet algebra $(\mathcal{A}(\D),\star_{\D})$. In particular,  a sequence $(f_n)$ in $\mathcal{A}(X)$ converges in the $\mathcal{A}(X)$--topology  to $f \in \mathcal{A}(X)$ if and only if $(f_n \circ \pi) \subseteq \mathcal{A}(\D)$ converges to $f\circ\pi \in \mathcal{A}(\D)$ in the $\mathcal{A}(\D)$--topology.
                  \end{definition}
\begin{remark} \label{rem:confinv}
      By conformal invariance of $(\A,\star_{\D})$, the actual choice of the universal
      covering $\pi : \D \to X$ is immaterial, and  the set $\mathcal{A}(X)$ as well as the product $\star_X$ are well--defined objects
      depending only on the Riemann surface $X$. Moreover, convergence of a sequence $(f_n) \subseteq \mathcal{A}(X)$ in the $\mathcal{A}(X)$--topology does not depend on the choice of $\pi$ either.   In addition, the star product $\star_X : \mathcal{A}(X) \times \mathcal{A}(X) \to \mathcal{A}(X)$ is continuous in the $\mathcal{A}(X)$--topology and $(\mathcal{A}(X),\star_X)$ is a Fr\'echet algebra.  These facts are not difficult to prove, but for convenience we will supply the details in Section \ref{sec:proofsnew}.  Since the star product $\star_{\D}$ depends on the choice of the deformation parameter $\hbar$,  the star product $\star_{X}$ on $X$ also depends on $\hbar$, and we write $\star_{\hbar,X}$ whenever the dependence on $\hbar$ is relevant.  However, the set $\mathcal{A}(X)$ is independent of $\hbar$.
\end{remark}

The main purpose of this note is to describe the Fr\'echet algebras
$(\mathcal{A}(X),\star_X)$ in an explicit way.
The following preliminary observation shows that $(\mathcal{A}(X),\star_X) $ depends
on the conformal type of $X$. Recall that two Riemann surfaces $X$ and $Y$ are
called conformally equivalent, if there is~a conformal (biholomorphic) map
$\Psi : Y \to X$.

      \begin{proposition}[Conformal invariance of the Fr\'echet algebra $(\mathcal{A}(X),\star_X)$] \label{prop:confinv}
Let $X$ and $Y$  be  hyperbolic Riemann surfaces. Then any conformal map
$\Psi : Y \to X$ induces  a Fr\'echet algebra isomorphism
$\Psi^* : (\mathcal{A}(X),\star_{X,\hbar}) \to (\mathcal{A}(Y),\star_{Y,\hbar}) $ defined by 
$$ \Psi^*(f)(z)= f(\Psi(z)) \, , \qquad f \in \mathcal{A}(X), \, z \in Y \, .$$
\end{proposition}

%In particular, if the Riemann surfaces $X$ and $Y$ are conformally equivalent,
%the algebras $(\mathcal{A}(X),\star_X)$ and $(\mathcal{A}(Y),\star_Y)$ are isomorphic.
The isomorphism $\Psi^*$ in Proposition \ref{prop:confinv}  does not depend on the deformation parameter $\hbar$, so $\mathcal{A}(X)$ and $\mathcal{A}(Y)$  are strongly isomorphic in the following sense.

\begin{definition}[Strong isomorphy]
  Let $X$ and $Y$ be Riemann surfaces. We call $\mathcal{A}(X)$ and $\mathcal{A}(Y)$ \textit{strongly isomorphic} if there is a map $T : \mathcal{A}(X) \to \mathcal{A}(Y)$ such that $T$ is a Fr\'echet algebra isomorphism between $(\mathcal{A}(X),\star_{X,\hbar})$ and $(\mathcal{A}(Y),\star_{Y,\hbar})$ for every $\hbar \in \mathscr{D}$.
In this case, we write  $\mathcal{A}(X) \si \mathcal{A}(Y)$. 
  \end{definition}

 We can now state  the main result of this paper.

\begin{theorem} \label{thm:main}
  Let $X$ be a hyperbolic Riemann surface. Then the following hold.
  \begin{itemize}
  \item[(a)] $X$ is conformally equivalent to $\D$ if and only if
    $$\mathcal{A}(X) \si \mathcal{A}(\D)\, .$$
   %   $$ \mathcal{A}(X) \si \A =\left\{z \mapsto F\left(z,\overline{z}\right) \, : \, F \in \mathcal{H}(\Omega) \right\}\, .$$
     \item[(b)] $X$  is conformally equivalent to some  annulus $A_R=\{z \in \C
       \,:\, 1/R<|z|<R\}$, $R>1$, if and only if $\mathcal{A}(X) \si \mathcal{A}(A_R)$. In this case,
       $$\mathcal{A}(X) \si  \mathcal{A}(A_{R}) \quad \text{ for every } R>1\, .$$
       
    \item[(c)] $X$ is conformally  equivalent to the punctured disk $\D^*:=\D
      \setminus \{0\}$ if and only   $$\mathcal{A}(X) \si \mathcal{A}(\D^*) \, .$$
      \item[(d)]  In all other cases,  $\mathcal{A}(X)$ consists only of constant functions.
    \end{itemize}
  \end{theorem}

\hide{  \begin{remarks} \textcolor{red}{verschieben$\rightarrow$ Beweis von Corollary \ref{cor:2}}
    Theorem \ref{thm:main} (a) is an immediate consequence of Theorem
    \ref{thm:KRSW} and Proposition \ref{prop:confinv}.
 Theorem \ref{thm:main} implies that $\mathcal{A}(X)$ separates points
if and only if $X$ is simply connected, since by (b) and (c) the
sets $\mathcal{A}(A_R)$ and $\mathcal{A}(\D^*)$ only consist of radially
symmetric functions, {\color{blue} and because $\Omega$ is a Stein manifold, see Theorem \ref{thm:G} (a).} 
\end{remarks}}

 The proof of  Theorem \ref{thm:main} depends in an essential way  on the  description
  of $\mathcal{A}(\D)$ in terms of $\mathcal{H}(\Omega)$ and  classical
  uniformization theory, and thereby on the complex
  analytic properties of the Riemann surface $X$. However,   the
  structure of $(\mathcal{A}(X),\star_X)$ as~a~Fr\'echet algebra  depends
  \textit{only} on the topological type of $X$, except when $X$ is doubly connected. This is the content of the following
  result, which follows from   Theorem \ref{thm:main} and the Uniformization Theorem.

    \begin{corollary} \label{cor:1}
 Let $X$ be a hyperbolic Riemann surface. Then the following hold.
  \begin{itemize}
  \item[(a)] $X$ is simply connected if and only if  $ \mathcal{A}(X) \si \A$.
        
  \item[(b)]  $X$ is doubly connected and conformally equivalent to some annulus $A_R$, $R>1$,  
    if and only if  $\mathcal{A}(X) \si \mathcal{A}(A_R)$. In this case, $\mathcal{A}(X) \si \mathcal{A}(A_R)$ for every $R>1$.

\item[(c)] $X$ is doubly connected and conformally equivalent to the punctured disk $\D^*$ if and only if
     $\mathcal{A}(X)  \si \mathcal{A}(\D^*) $.

\item[(d)]  $X$ is neither simply nor doubly connected if and only if  $\mathcal{A}(X)$ consists only of constant functions.
    \end{itemize}
      \end{corollary}

      In particular, if $X$ is a compact Riemann surface of genus greater or equal than two, then $\mathcal{A}(X)$  degenerates and consists only of constant functions. This might be viewed as a  very strict  topological rigidity phenomenon. 
      We also see that the  Fr\'echet algebras $(\mathcal{A}(A_R),\star_{A_R})$ with  $R>1$  are all strongly isomorphic, even though  $A_R$ and $A_{R'}$  are not conformally equivalent if $R\not=R'$. The following result gives an explicit description of $\mathcal{A}(A_R)$ and the star product $\star_{A_R,\hbar}$. In order to properly formulate this result, we introduce for each $R>0$ the  auxiliary function
      \begin{equation} \label{eq:f_RDef}
  f_R : A_R \to \C \, , \qquad f_R(z):= -i \tan \left( \frac{\pi}{2 \log R} \log |z| \right) \, .
\end{equation}   
The choice of the prefactor $-i$ in (\ref{eq:f_RDef}) facilitates some of the computations.  We also denote
\begin{equation} \label{eq:coeff}
  c_n(\hbar):=\begin{cases} \, \displaystyle  \frac{\hbar^n}{\prod \limits_{j=0}^{n-1} \left(1+j \hbar \right)} & \text{ for } n \ge 1 \, , \\
    \, 1 & \text{ for } n=0 \, .
    \end{cases}
    \end{equation}

      \begin{theorem}[The Fr\'echet algebra $(\mathcal{A}(A_R),\star_{A_R})$ for annuli $A_R=\{z \in \C \, : \, 1/R<|z|<R\}$] \label{thm:mainannulus}
        Let $R>1$. Then the mapping
\begin{equation} \label{eq:T_R}
  T_R : \mathcal{H}(\C) \to \mathcal{A}(A_R) \, , \qquad T_R(g):=g \circ f_R\, ,
  \end{equation}
  is a continuous linear bijection between the Fr\'echet spaces $\mathcal{H}(\C)$ and $\mathcal{A}(A_R)$.
In particular, 
\begin{equation} \label{eq:A(A_R)}
\mathcal{A}(A_R)=\big\{ g \circ f_R \, : \,  g \in \mathcal{H}(\C)\big\} \, .
\end{equation}
 
Moreover, let $g, \tilde{g} \in \mathcal{H}(\C)$, and  $f:=g \circ f_R, \tilde{g} \circ f_R \in \mathcal{A}(A_R)$. Then 
 \begin{equation} \label{eq:FormulaStarAnnulus}
          f \star_{A_R} \tilde{f}=\sum \limits_{n=0}^{\infty} \frac{c_n(\hbar)}{n!} \left( f_R^2-1 \right)^n  \left( g^{(n)} \circ f_R \right) \cdot  \left( \tilde{g}^{(n)} \circ f_R \right) \, .
        \end{equation}
        In particular, the star product $\star_{A_R}$ is commutative,
          $$f \star_{A_R} \tilde{f}=\tilde{f} \star_{A_R} f\, .$$
Furthermore, for fixed $\hbar \in \mathscr{D}$, the series (\ref{eq:FormulaStarAnnulus}) converges in the topology of $\mathcal{A}(A_R)$.
     
          \end{theorem}

          As a corollary we deduce that the Fr\'echet algebras $(\mathcal{A}(A_R),\star_{A_R})$ and  $(\mathcal{A}(A_{R'}),\star_{A_{R'}})$ are strongly isomorphic for every choice of $R>1$ and $R'>1$:
          
\begin{corollary} \label{cor:mainannulus}
  Let $R,R' >1$. Then the mapping
  $$\Psi_{R',R} : \mathcal{A}(A_{R'}) \to \mathcal{A}(A_R) \, , \qquad \Psi_{R',R}:=T_R \circ T^{-1}_{R'} $$ 
is a Fr\'echet algebra isomorphism between $(\mathcal{A}(A_{R'}),\star_{A_{R'},\hbar})$ and $(\mathcal{A}(A_R),\star_{A_R,\hbar})$ for every $\hbar \in \mathscr{D}$.
\end{corollary}

The next result  gives an explicit description of $\mathcal{A}(\D^*)$ and the star product $\star_{\D^*,\hbar}$ for the punctured unit disk $\D^*=\D\setminus \{0\}$. We first introduce the auxiliary function
\begin{equation} \label{eq:f_0Def}
  f_0 : \D^* \to \C \, , \qquad f_0(z):=-\frac{1}{\log |z|}  \, .
  \end{equation}

        \begin{theorem}[Wick star product for the punctured disk $\D^*=\D \setminus \{0\}$] \label{thm:mainpunctureddisk}
          The mapping
           $$T_0 : \mathcal{H}(\C) \to \mathcal{A}(\D^*) \, , \qquad T_0(g):=g \circ f_0 $$
        is a  continuous linear bijection between the Fr\'echet spaces $\mathcal{H}(\C)$ and $\mathcal{A}(\D^*)$.
        In particular,  
        \begin{equation} \label{eq:A(D^*)}
          \mathcal{A}(\D^*)=\big\{ g \circ f_0 \, : \,  g \in \mathcal{H}(\C) \big\} \, .
          \end{equation}
          Moreover, let  $g, \tilde{g} \in \mathcal{H}(\C)$, and  $f:=g \circ f_0, \tilde{g} \circ f_0 \in \mathcal{A}(\D^*)$. Then
          \begin{equation} \label{eq:FormulaStarPuncturedDisk}
          f \star_{\D^*} \tilde{f}=\sum \limits_{n=0}^{\infty} \frac{c_n(\hbar)}{n!} \, f_0^2 \cdot \left( g^{(n)} \circ f_0 \right) \cdot \left( \tilde{g}^{(n)} \circ f_0 \right) \, .
        \end{equation}
        In particular, the star product $\star_{\D^*}$ is commutative, $$f \star_{\D^*} \tilde{f}=\tilde{f} \star_{\D^*} f \, .$$
Furthermore, for fixed $\hbar \in \mathscr{D}$, the series (\ref{eq:FormulaStarPuncturedDisk}) converges in the topology of $\mathcal{A}(\D^*)$.
\end{theorem}

\begin{remark}
The series in (\ref{eq:FormulaStarAnnulus}) as well as in (\ref{eq:FormulaStarPuncturedDisk}) is a \textit{factorial series} in $\hbar \in \mathscr{D}$ and converges locally uniformly in $\mathscr{D}$.
\end{remark}

The following result was a bit  surprising for us.

\begin{theorem} \label{cor:mainpunctureddisk}
 Let  $R>1$. Then the commutative Fr\'echet algebras $(\mathcal{A}(A_R),\star_{A_R})$ and $(\mathcal{A}(\D^*),\star_{\D^*})$ are  not strongly isomorphic.
  \end{theorem}

The proof of Theorem \ref{cor:mainpunctureddisk} is based on the explicit formulas (\ref{eq:FormulaStarAnnulus}) for $\star_{\hbar,A_R}$ and (\ref{eq:FormulaStarPuncturedDisk}) for $\star_{\hbar,\D^*}$, but also on an asymptotic series for $\star_{\hbar,\D}$ which has recently been established in \cite[Theorem~6.9]{HeinsMouchaRoth1}.

  \begin{remark}\label{Z2_symmetry}  As it has been observed in \cite[Section 4.5]{KRSW}, the algebra $\mathcal{A}(\D)$ possesses 
	a ``$\mathbb{Z}_2$--symmetry'': if $f\in \mathcal{A}(\D), f(z)=g(z,\overline{z})$ with $g\in \mathcal{H}(\Omega)$, then 
	$g(1/z,1/w)$ also belongs to $\mathcal{H}(\Omega)$. This implies that $f$ can be extended in a natural way to $\hat{\C}\setminus \overline{\D}$ and that 	$\D\ni z\mapsto f(1/z)$ also belongs to $\mathcal{A}(\D)$.
	This symmetry is carried over to $\mathcal{A}(\D^*)$ and $\mathcal{A}(A_R)$:
        \begin{itemize}
          \item[(a)] 
	If $f=g\circ f_0$ is in  $\mathcal{A}(\D^*)$, then $f$ can be extended to $\C \setminus \overline{\D}$ and 
	$\tilde{f}(z):=f(1/z)$ again belongs to $\mathcal{A}(\D^*)$.
        \item[(b)]
	If $f=g\circ f_R $ is in $\mathcal{A}(A_R)$, then 
        $f$ extends naturally to $\{z\in\hat{\C}\,:\, |z|\not = R^{2k+1}, k\in \mathbb{Z}\cup\{\pm\infty\}\}$.
        \end{itemize}
        Thus,  in all three cases, $f$ extends to the sphere $\hat{\C}$ minus a set of circles (including radius zero).

        \medskip

We can factor out the $\mathbb{Z}_2$--symmetry and consider the set $$\mathcal{A}'(\D)=\left\{g(z,\overline{z}) \,:\,g\in \mathcal{H}(\Omega), 
g(z,w)=g\left(\frac{1}{z},\frac1{w}\right)\right\}\, .$$ Then $\mathcal{A}'(\D)$ is a $\star_{\D}$-invariant subset of $\mathcal{A}(\D)$.
Factoring out the symmetry $z\mapsto 1/z$ for $\mathcal{A}(\D^*)$ simply gives the $\star_{\D^*}$-invariant set 
$$\mathcal{A}'(\D^*)=
\left\{ g \circ f_0^2 \,  : \, g \in \mathcal{H}(\C) \right\}.$$
\end{remark}

Theorem \ref{thm:main}, Theorem \ref{thm:mainannulus} and Theorem \ref{thm:mainpunctureddisk} imply the following characterizations of simply connected hyperbolic Riemann surfaces in terms of properties of $\mathcal{A}(X)$.

     \begin{corollary} \label{cor:2}
    Let $X$ be a hyperbolic Riemann surface. Then the following are equivalent.
    \begin{itemize}
    \item[(a)]    $X$ is simply connected.
    \item[(b)] $\mathcal{A}(X)$ separates points.
    \item[(c)] The Fr\'echet algebra $(\mathcal{A}(X),\star_X)$ is noncommutative.
    \end{itemize}
  \end{corollary}
  
      \begin{remark}[Parameter space of $\mathcal{A}(X)$]
Combining Theorem \ref{thm:main} and Corollary \ref{cor:1}, and noting  that
$\mathcal{H}(\Omega)$ is isomorphic (as a vector
space) to the space $\mathcal{H}(\C^2)$ of entire functions on $\C^2$, see
\cite{KRSW} as well as \cite[Section 4]{HeinsMouchaRoth3}, it is tempting to think of   the Fr\'echet space
$\mathcal{A}(X)$ as being  ``parametrized'' by the Fr\'echet space
$\mathcal{H}(\C^k)$ with $k$ chosen in such a way that  $X$ is
$\max\{3-k,0\}$--connected:

\renewcommand{\arraystretch}{1.5}
\begin{table}[h]
  \begin{center}
\begin{tabular}{c|c|c}
Connectivity of $X$  & Parameter space of $\mathcal{A}(X)$  & Star product $\star_X$ \\ \hline 
  \quad$1$ &   $\mathcal{H}(\C^2)$ & noncommutative \\
  \quad $2$  &  $\mathcal{H}(\C^1)$ & commutative \\
  $\ge 3$ &  \hspace{0.7cm} $\mathcal{H}(\C^0) \simeq \C$ & trivial
  \end{tabular}\end{center} \caption{Topology of $X$ vs.~''Dimension'' of the
  Parameter space}
\end{table}
\end{remark}

The plan of the present paper is as follows. In a preliminary Section \ref{sec:proofsnew} we prove the basic properties of the star poduct stated  in Proposition \ref{prop:1}, Remark \ref{rem:confinv}  and Proposition \ref{prop:confinv}. These properties will be needed throughout the paper.
In Section \ref{sec:Conf} we discuss and prove a number of topological and complex analytic properties of the manifold $\Omega$. It turns out to be more convenient to pass from $\Omega$   to the biholomorphically equivalent
manifold $$G:=\hat{\C}^2 \setminus \{ (z,z) \, : z \in \hat{\C}\} \, .$$
The set $G$ is a specific example of a configuration space and is also called  \textit{second ordered configuration
  space of the sphere}.
Such spaces  arise in various problems in algebraic topology, knot theory and mathematical physics, see e.g.~\cite{Coh}. Section \ref{sec:Automorphic} is partly based on Section~\ref{sec:Conf} and is concerned with holomorphic \textit{automorphic} functions on $G$. We obtain explicit representations of all functions in $\mathcal{H}(G)$ which are invariant w.r.t.~to certain discrete subgroups of  automorphisms of~$G$. More precisely, those subgroups are induced by  fixed--point free Fuchsian groups $\Gamma$ of automorphisms of the upper half--plane $\H$,  which by classical Riemann surface theory represent the category of all hyperbolic Riemann surfaces. The results of Section \ref{sec:Automorphic}  are among the main technical contributions of this paper; their proofs are given in Section \ref{sec:proofsMoebius}.
Section \ref{sec:Doublyconnected} is devoted to the exceptional case of doubly connected Riemann surfaces, and we prove  Theorem \ref{thm:mainannulus} and  Corollary \ref{cor:mainannulus} as well as Theorem \ref{thm:mainpunctureddisk}  and Theorem \ref{cor:mainpunctureddisk}. In Section~\ref{sec:proofs} we finally conclude with the proof of  Theorem \ref{thm:main} and  Corollary \ref{cor:2}.

\section{Basic properties of the star products} \label{sec:proofsnew}

  %\subsection{Proof of Proposition \ref{prop:1}, Remark \ref{rem:confinv}  and Proposition \ref{prop:confinv}} \label{subsec:3_1}

In this section we establish the basic properties of the continuous Wick star product on the unit disk and hyperbolic Riemann surfaces. In particular, we prove  Proposition \ref{prop:1}, Remark \ref{rem:confinv}  and Proposition \ref{prop:confinv}.

\medskip

Let us first make clear  that for each $f \in \A$ there is a unique $F \in \mathcal{H}(\Omega)$ such that $f(z)=F(z,\overline{z})$ for all $z \in \D$.  In fact, let $f \in \mathcal{A}(\D)$ and suppose $F,\tilde{F} \in \mathcal{H}(\Omega)$ such that $F(z,\overline{z})=f(z)=\tilde{F}(z,\overline{z})$ for all $z \in \D$.
Hence  a well--known variant of the identity principle (\cite[p.~18]{Range}) implies $F=\tilde{F}$, first on $\D \times \D$ and then by connectedness on the entire manifold~$\Omega$.

\begin{proof}[Proof of Proposition \ref{prop:1}]
This follows from the construction of the algebra $(\A,\star_{\D})$ in \cite{KRSW}.
In fact, the conformal invariance of the set $\A$ also follows directly from
Theorem \ref{thm:KRSW}. In order to see this let $f \in \mathcal{A}(\D)$ and  $\varphi \in \Aut(\D)$. By definition, there is $F \in \mathcal{H}(\Omega)$ such that $f(z)=F(z,\overline{z})$ for all $z \in
\D$. Since $\varphi$ is a M\"obius
transformation which maps $\hat{\C}$ bijectively onto $\hat{\C}$,
we immediately see that
$$T_{\varphi}(z,w):=\left(\varphi(z),\frac{1}{\varphi(1/w)} \right)$$
defines a biholomorphic selfmap of $\Omega$. Now, in  view of  the symmetry
property $\overline{\varphi(z)}=1/\varphi(1/\overline{z})$, we get
$$ (f \circ
\varphi)(z)=F\left(\varphi(z),\overline{\varphi(z)}\right)=F(\varphi(z),1/\varphi(1/\overline{z}))
=(F \circ T_{\varphi})(z,\overline{z})$$
for all $z \in \D$.  Since  $F \circ T_{\varphi} \in H(\Omega)$,
we deduce  $f \circ \varphi \in \A$.
\end{proof}

 We can now easily see that the definition of the  algebra $(\mathcal{A}(X),\star_X)$ does not depend
on the choice of the universal covering map $\pi : \D \to X$ and that  $(\mathcal{A}(X),\star_X)$ is a Fr\'echet algebra w.r.t.~the topology inherited from $(\mathcal{A},\star_X)$.

\begin{proof}[Proof of Remark \ref{rem:confinv}]%[$(\mathcal{A}(X),\star_X)$ is well defined] \label{rem:ok}
 Let $X$ be a
hyperbolic Riemann surface and $\pi,\tilde{\pi} : \D \to X$ universal covering
maps. Then $\pi=\tilde{\pi} \circ \varphi$ for some $\varphi \in
\Aut(\D)$. Proposition
    \ref{prop:1} shows that if $f : X \to \C$ is a function, then  $f \circ \pi \in \A$ implies $f \circ
    \tilde{\pi} \in \A$. Therefore, the set $\mathcal{A}(X)=\{f : X \to \C \, |
    \, f \circ \pi \in \A\}$ does not depend on the choice of $\pi$.
 Now let $\Gamma$ denote the covering group of $X$ w.r.t.~the covering $\pi :
 \D \to X$, so $\pi  \circ \gamma=\pi$ for any $\gamma \in \Gamma$. Let $p
 \in X$ and $z_0,z_1 \in \D$ such that $\pi(z_0)=\pi(z_1)=p$. Since $\Gamma$
 acts transitively on the fiber $\pi^{-1}(\{p\})$  there is an 
    $ \gamma \in \Gamma$ such that $\gamma(z_0)=z_1$. This implies
   \begin{eqnarray*} \left[ (f \circ \pi) \star_{\D} (g \circ \pi)\right] (z_1)&=&
    \left[ \left( (f \circ \pi) \star_{\D} (g \circ \pi)\right) \circ \gamma \right] (z_0)=
   \left[ (f \circ \pi \circ \gamma ) \star_{\D} (g\circ \pi \circ \gamma)\right](z_0)
   \end{eqnarray*}
   by Proposition \ref{prop:1}. In view of $\pi \circ \gamma=\pi$, we see that $(f \star_X g)(p)$ is well--defined. If $\tilde{\pi} : \D \to X$ is another universal covering, then $\pi=\tilde{\pi} \circ \varphi$ for some $\varphi \in \Aut(\D)$ and Proposition
   \ref{prop:1} implies
   $$ \left[ (f \circ \pi) \star_{\D} (g \circ \pi) \right] =
 \left[ (f \circ \tilde{\pi} \circ \varphi) \star_{\D} (g \circ \tilde{\pi} \circ \varphi) \right] =
\left[ (f \circ \tilde{\pi}) \star_{\D} (g \circ \tilde{\pi})  \right] \circ \varphi .
   $$
   Hence $ f \star_X g : X \to \C$  does not depend on the choice of $\pi$.
\medskip

   We next show that convergence in $\mathcal{A}(X)$ does not depend on the choice of $\pi$ either.
Let $(f_n) \subseteq \mathcal{A}(X)$ and $f\in \mathcal{A}(X)$ such that $f_n \circ \pi \to f \circ \pi$ in the $\mathcal{A}(\D)$--topology. This means that there are functions $F_n,F \in \mathcal{H}(\Omega)$ such that $(f_n \circ \pi)(z)=F_n(z,\overline{z})$ and $(f \circ \pi)(z)=F(z,\overline{z})$ for all $z \in \D$ and $F_n \to F$ locally uniformly in $\Omega$. If $\tilde{\pi} : \D \to X$ is another universal covering, then $\tilde{\pi}=\pi \circ \gamma$ for some $\gamma \in \Aut(\D)$. As in the proof of Proposition \ref{prop:1},
$$ T_{\gamma}(z,w):=\left( \gamma(z), \frac{1}{\gamma(1/w)} \right)$$
is a biholomorphic selfmap of $\Omega$, and $F_n \circ T_{\gamma}$ converges locally uniformly in $\Omega$ to $F \circ T_{\gamma}$. By definition, this means  $f_n \circ \tilde{\pi} \to f \circ \tilde{\pi}$ in the $\mathcal{A}(\D)$--topology.

\medskip

   Finally, let us convince ourselves that $\mathcal{A}(X)$ is a Fr\'echet space such that $\star_{X} : \mathcal{A}(X) \times \mathcal{A}(X) \to \mathcal{A}(X)$ is continuous. Let $(f_n)$ be a Cauchy sequence in $\mathcal{A}(X)$. This means that there are $F_n \in \mathcal{H}(\Omega)$ such that $g_n(z):=(f_n\circ\pi)(z)=F_n(z,\overline{z})$ on $\D$ and $(F_n)$ is Cauchy in $\mathcal{H}(\Omega)$.
   By completeness of $\mathcal{H}(\Omega)$ there is $F \in \mathcal{H}(\Omega)$ such that $F_n \to F$ locally uniformly in $\Omega$. Then $g(z):=F(z,\overline{z})$ belongs to $\mathcal{A}(\D)$ by definition, and $g_n \to g$ locally uniformly in $\D$.
   Now fix a point $x_0 \in X$, and choose $z_0 \in \D$ such that $\pi(z_0)=x_0$. It is easy to see that $g(z_0)$ only depends on $x_0$, but not on $z_0$. Hence, we can define $f(x_0):=g(z_0)$ and get a function $f : X \to \C$ such that $f \circ \pi=g$ and $f_n \to f$ in the $\mathcal{A}(X)$--topology. Thus $\mathcal{A}(X)$ is a Fr\'echet space. The continuity of  $\star_{X} : \mathcal{A}(X) \times \mathcal{A}(X) \to \mathcal{A}(X)$ follows at once from the definitions. 
 \end{proof}    

\begin{proof}[Proof of Proposition \ref{prop:confinv}]
Let $X$ and $Y$ be hyperbolic Riemann surfaces, $\pi : \D \to
Y$  a universal covering of $Y$  and $\Psi : Y\to X$  a conformal map. Then $\Psi
\circ \pi : \D \to X$ is a universal covering of $X$. Hence we have for
$\Psi^*: \mathcal{A}(X) \to \mathcal{A}(Y)$, $\Psi^*(f)=f \circ \Psi$ that
$$ \Psi^*(\mathcal{A}(X))=\{ \Psi^*(f) : X \to \C \, | \, f \circ \Psi \circ
\pi \in \A\}=\{ g : Y \to \C \, : g \circ \pi \in A\}=\mathcal{A}(Y)\,.  $$
Moreover,
\begin{eqnarray*}
 \Psi^*(f \star_X g) \circ \pi&=& (f \star_X g) \circ \Psi \circ \pi=
(f \circ \Psi \circ \pi) \star_{\D} (g \circ \Psi \circ \pi)\\ &=& \left(\Psi^*(f) \circ
\pi\right) \star_{\D} \left(\Psi^*(g) \circ \pi\right)= \left(\Psi^*(f)
                                                             \star_Y \Psi^*(g)
                                                                   \right)
                                                                   \circ \pi
                                                                   \, ,
\end{eqnarray*}
so $\Psi^* : (\mathcal{A}(X),\star_X) \to
  (\mathcal{A}(Y),\star_Y)$ is an algebra isomorphism.
Finally, it is to see that $\Psi^* : (\mathcal{A}(X),\star_X) \to
  (\mathcal{A}(Y),\star_Y)$ is continuous.
\end{proof}

\section{Complex structure  of the second ordered configuration space of the sphere} \label{sec:Conf}

For some purposes of this paper, it is somewhat more convenient to work on the upper
half--plane~$\H$ instead of the unit disk $\D$. 
Since every conformal map~$T$ from $\D$ onto $\H$ is actually~a M\"obius
transformation  which maps $\hat{\C}$ biholomorphically onto itself,  we are in a position to replace the complex manifold
$\Omega$ by its image under the map 
$$ (z,w) \mapsto \Psi(z,w):= \left( T(z),T(1/w)  \right)
\, .$$
Note that $\Psi$  maps  $\Omega$ biholomorphically onto
$$ G:=\Psi(\Omega)=\hat{\C}^2 \setminus \{ (z,z) \, : \, z \in \hat{\C} \} \, ,$$
and by conformal invariance we have
$$ \mathcal{A}(\H)=\big\{f : \H \to \C \, | \, f(z)=F(z,\overline{z}) \text{ for all } z \in \H \text{ for some } F \in \mathcal{H}(G)\big\} \, .$$

The set $G$ is also called  the second  configuration space
 of the sphere $\hat{\C}$ and is denoted by $\text{Conf}_2(\hat{\C})$, see \cite{Coh}.
Its fundamental group $\pi_1(G)$ is the $2$--strand pure braid group of $\hat{\C}$.

\begin{theorem} \label{thm:G}
  $G$ is a connected simply connected complex manifold. In addition, $G$ 
  \begin{itemize}
\item[(a)]  is a Stein manifold;
\item[(b)]  is not  biholomorphic to $\C^2$ but contains an open subset which is biholomorphic to $\C^2$;
%\item[(c)] {\color{blue}is biholomorphic to the complexified 2-sphere $S^2_\C := \{(z_1,z_2,z_3)\in\C^3\,|\, z_1^2+z_2^2+z_3^2=1\}$;
\item[(c)] is diffeomorphic to the real tangent bundle $T_\R S^2$ of the 2-sphere $S^2\subset \R^3$;
\item[(d)]  possesses the density property from Anders\'{e}n-Lempert theory. 
     \end{itemize}
    \end{theorem}
    
    \begin{remark} With regard to Theorem \ref{thm:G} (b), we note that one can directly show that 
     the open subsets $G_{+}:=G \setminus \{(z,0) \, : z \in \hat{\C}\}$ and $G_-:=G \setminus \{(0,w) \, : w \in \hat{\C}\}$ of $G$ are biholomorphically equivalent to $\C^2$. This is proved in \cite[Proposition 4.3]{HeinsMouchaRoth3}
     for $\Omega$ instead of $G$. In the original $\Omega$--model the set $G_+$ corresponds to $\Omega_+=\Omega \setminus \{(z,\infty) \, : \, z \in \hat{\C}\}$ and $G_-$ to $\Omega_-:=\Omega \setminus \{(\infty,w) \, : \, w \in \hat{\C}\}$.
     The sets $\Omega_+$ and $\Omega_-$ play a crucial role for the Fr\'echet space structure of $\mathcal{H}(\Omega)$ (see \cite{HeinsMouchaRoth3}), for the study of invariant differential operators on $\mathcal{H}(\Omega)$ (see \cite{HeinsMouchaRoth1}) and for the spectral theory of the invariant Laplacian on $\Omega$ (see \cite{HeinsMouchaRoth2,Annika}).   
      \end{remark}
    
\begin{remark} \label{rem:HomSur}
            In complex geometry, the manifold $G$ often appears in the biholomorphically equivalent form of the complex two--sphere  $$S^2_\C := \{(z_1,z_2,z_3)\in\C^3\,|\, z_1^2+z_2^2+z_3^2=1\} \, ,$$ and thus 
            as a complex homogeneous surface (see the classification \cite{McKay}), or as an entire Grauert tube (see \cite{To}).
A coordinate change shows that $G$ is also biholomorphic to
\begin{equation} \label{eq:Dani}
  D:=\left\{(a,b,c)\in \C^3\,:\, b^2-4ac = 1\right\} \, ,
\end{equation}
  which is  known as a \textit{Danielewski surface}  
when written in the form
$\{(a,b,c)\in \C^3\,:\, ac = (b^2-1)/4\}$. (A Danielewski surface has the form $\{(x,y,z)\in\C^3\,:\, x^n y=p(z)\}$ for a polynomial $p$ with simple roots.)

In fact,  the mapping
 \[ f : G \to D \, , \quad (z,w) \mapsto \left(\frac{1}{z-w}, \frac{z+w}{z-w}, \frac{zw}{z-w}\right).\]
  is biholomorphic (\cite[Section 10.1]{McKay}.
The change of coordinates 
%\begin{equation}\label{ch_co} z_1:=\frac{i\sqrt{2}}{2}a+\frac{i\sqrt{2}}{2}c,\quad z_2=b,\quad z_3=\frac{-\sqrt{2}}{2}a+\frac{\sqrt{2}}{2}c.\end{equation}
\begin{equation}\label{ch_co} z_1:=ia+ic,\quad z_2:=b,\quad z_3:=-a+c.\end{equation}
% diagonalize {{0,0,-2},{0,1,0},{-2,0,0}}
 yields the homogeneous quadric $S_\C^2$.
\end{remark}
    
The characterization of $G$ in Theorem \ref{thm:G} (c) suggests that our quantization is related to a quanitzation of the 2-sphere $S^2$, where $S^2$ is regarded as the configuration space and 
$T_\R S^2$ as the phase space. Indeed, certain holomorphic functions on $S^2_\C$ appear in this context in \cite{HM} (with corrections \cite{HM2, RJ}), see \cite[Theorem 5]{HM} for the space of holomorphic functions. 
This space is conformally invariant in the sense of \cite[Lemma 3]{HM}.\\
Our construction is related to the 2-sphere insofar as the symmetry of $\mathcal{A}(\D)$ with respect to the transform $z\mapsto 1/z$ (see \textcolor{red}{Remark \ref{Z2_symmetry}}) 
shows that we are actually considering functions that are defined on two discs at the same time, the unit disc $\D$ and the disc $\hat{\C}\setminus \overline{\D}$, whose union 
gives the 2-sphere $\hat{\C}$ (minus $\partial \D$).

%The transfrom $(z,w)\mapsto (1/z,1/w)$ on $\Omega$ becomes $(z,w)\mapsto (-z,-w)$ on $G$. Let $p:\hat{\C}\to \hat{\C}/(z\sim -z)$ be the natural projection. Then 
%$p$ embeds the disc $\H$ diffeomorphically into $p(\hat{\C})$, which is diffeomorphic to the real projective plane $\R P^2$. The holomorphic functions $\mathcal{A}'(\D)$ can be identified with 
%$\mathcal{H}(G')$, where $G'=G/((z,w)\sim (-z,-w))$ is diffeomorphic to ??

\begin{remark}\label{aut_omega}Theorem \ref{thm:G} (d) implies that $\Aut(G)$ is infinite-dimensional. This can also be seen directly by verifying that 
the following mappings belong to $\Aut(G)$: 
\[ (z,w)\mapsto \left(\alpha z + g\left(\frac1{z-w}\right), \alpha w + g\left(\frac1{z-w}\right)\right),\]
 where $\alpha\in \C^*$ and $g:\C\to \C$ is an entire mapping. 
 Furthermore, $\Aut(G)$ also contains the mappings
\begin{equation} \label{eq:autG1}
 (z,w)\mapsto (w,z), \quad (z,w)\mapsto \left(\varphi(z), \varphi(w)\right), \quad \varphi\in\Aut(\hat{\C}).
    \end{equation}
 In \cite[Theorem 5.1]{HeinsMouchaRoth3} it is shown that these automorphisms are the only ones in the infinite--dimensional group $\Aut(G)$ which commute with the action of the natural Laplacian on $G$. Even though we do not need this result, it partly explains why we only consider automorphisms on $G$ of the form (\ref{eq:autG1}).
\end{remark}

\begin{remark}
The complex (compact K\"ahler)  manifold $\hat{\C} \times \hat{\C}$ can be embedded into 
$\C P^3$ via Segre's embedding $\varphi([z_0:z_1],[w_0:w_1])= [z_0w_0:z_0w_1:z_1w_0:z_1w_1]$.
We have $\varphi(\hat{\C}\times\hat{\C})=\{[z_0:z_1:z_2:z_3]\,|\, z_0z_3-z_1z_2=0\}$ and thus we obtain yet another way to represent $G$ as 
 \[\varphi(G)=\{[z_0:z_1:z_2:z_3]\,|\, z_0z_3-z_1z_2=0, z_1\not=z_2\}.\]
 In this context we can consider the group of all rational automorphisms of $G$, which turns out to be generated by all automorphisms from Remark \ref{aut_omega} with $g$ being a polynomial,
 see \cite[Theorem 4]{6}.
\end{remark}

\begin{remark}
 An automorphism $\tau$ of a complex manifold is called a ``generalized translation'' if for any holomorphically convex compact $K\subset G$ there exists $m\in \N$ with 
 $\tau^m(K)\cap K = \emptyset$ and $\tau^m(K)\cup K$ is holomorphically convex.  A Stein manifold with density property which has a generalized translation possesses universal automorphisms in the sense of \cite[Theorem 3.6]{AW}. 
 Our manifold $G$ has these three properties. A generalized translation (for certain Danielewski surfaces) is given in \cite[Example 4.3]{AW}. On $G$, this simply corresponds to 
 $\tau(z,w)=(z+\alpha, w+\alpha),$ $\alpha\in\C^*$.
\end{remark}

\begin{remark}Consider the set of all $g\in \mathcal{A}(\D)$ with $g(z)=g(\overline{z})$, which corresponds to all $F\in\mathcal{H}(\Omega)$ with $F(z,w)=F(w,z)$. 
As $M:=\Omega/((z,w)\sim (w,z))$ is again a complex manifold, the latter set simply corresponds to $\mathcal{H}(M)$. 
With \cite[Section 10.2]{McKay} and the change of coordinates from \eqref{ch_co} wee see that $M$ is biholomorphic to
\[\{[z_1:z_2:z_3]\in \C P^2\,:\, z_1^2+z_2^2+z_3^2\not=0\}.\]
\end{remark}
    
 \begin{proof}[Proof of Theorem \ref{thm:G}]
$G$ is an open subset of $\hat{\C}^2$ and pathconnected and hence connected.
   It is known that $G$ is homotopy equivalent to $\hat{\C}$, see \cite[Example
  2.4 (1)]{Coh}. Thus $\pi_1(G)$ is trivial and $G$ is simply connected.
\smallskip

(a) Using functions $f:\Omega\to\C$ of the form $f(z,w)=g(1/(\gamma(z)-\gamma(w))$ with  $\gamma$ as a Moebius transform and $g\in\mathcal{H}(\C)$, it is straightforward to check that $G$ is Stein (see \cite[Def.~2.2.1]{7}).
This also follows from Remark \ref{rem:HomSur} since closed submanifolds of $\C^n$ are always Stein.

\smallskip

(b) Since $G$ is homotopy equivalent to $\hat{\C}$, it is not biholomorphic to $\C^2$.
It is is known that there exist Fatou--Bieberbach domains in $\C^2$ avoiding a complex line, see \cite[Example 4.3.9]{7}, so $G \cap \C^2$ contains a Fatou--Bieberbach domain
(i.e.~a proper subset of $\C^2$ which is biholomorphically equivalent to $\C^2$).
This statement also follows more generally from (a), (d), and \cite[Corollary 4.1]{Va}.

\smallskip

%(c) This is clear since by Riemann's theorem about removable singularities every bounded function in $\mathcal{H}(G)$ has a holomorphic extension to $\hat{\C} \times \hat{\C}$ and is therefore constant.

\smallskip

(c) The diffeomorphic equivalence between the real tangent bundle of a real analytic manifold and its complexification is a more general fact in the sense of \cite[p.104]{CE}. 
An explicit diffeomorphism for our case can be found in \cite[p.410]{Sz}.

\smallskip

(d) Remark \ref{rem:HomSur} shows that $G$ can be represented as $\{(a,b,c)\in\C^3\,:\, ac=(b^2-1)/4\}$. The density property now follows from the main theorem in \cite{KK} with $p(x)=(x^2-1)/4$.
 \hide{
% Denote by $S^2\subset \R^3$ the 2-sphere in $\R^3$. Clearly, $G$ is diffeomorphic to $H:=(S^2\times S^2)\setminus\{(p,p)\,:\, p\in S^2\}$.
% For $p\in S^2$, let $\text{prj}_{p}(q):S^2\setminus\{p\}\to \R^3$ be the stereografic projection of $S^2\setminus\{p\}$ onto the tangent plane of $S^2$ at the point $-p$. 
% Then $\text{prj}_{p}(q)+2p$ can be identified with an element of $T_p S^2$.
% 
% It be verified that the following function $f:H\to T_\R S^2=\{(x,v)\,:\, x\in S^2, v\in T_xS^2\}$ is a diffeomorphism: 
% 
% \[ (p,q)\mapsto \left(p, \text{prj}_{p}(q)+2p\right). \]
}
\end{proof}

\section{M\"obius--invariant holomorphic functions on the second configuration space of the sphere} \label{sec:Automorphic}

In order to  describe the algebras $\mathcal{A}(X)$ for hyperbolic Riemann
surfaces $X$ which are not simply connected, we are going to use  some few basic facts about
Fuchsian groups, see e.g.~\cite{Hub,Th,Katok}.
Recall that a \textit{Fuchsian group} is a discrete M\"obius group $\Gamma$ acting  on the upper half--plane $\H$. In particular,
every Fuchsian group  is a subgroup of the group $\Aut(\H)$ of all conformal
automorphisms of $\H$  and
the  quotient  $\H/\Gamma$ is a hyperbolic Riemann surface.
In fact, every hyperbolic Riemann surface arises in this way. 
If $X$ is a hyperbolic Riemann surface and $\pi : \H \to X$ is a universal
covering map, then its group $\Gamma$ of covering transformations is a
\textit{fixed point free} Fuchsian group, that it, every $\gamma \in \Gamma
\setminus \{id\}$ has no fixed points \textit{in} $\H$. In this case, $X$ is
conformally equivalent  to $\H/\Gamma$, and every  function $ g : X \to \C$  can be identified with the
$\Gamma$--invariant function $f:=g \circ \pi : \H \to \C$ and vice versa.

\medskip

We conclude that in order to find $\mathcal{A}(\H/\Gamma)$ we only need to identify the $\Gamma$--invariant functions in $\mathcal{A}(\H)$.
Now every $\gamma \in \Aut(\H)$ induces in a natural way an element $\hat{\gamma}$ of the group $\Aut(G)$ of biholomorphic self--maps of $G$ via $$\hat{\gamma}(z,w)=(\gamma(z),\gamma(w))\, .$$

\begin{lemma} \label{lem:2_1}
  Let $f \in \mathcal{A}(\H)$ and $F \in \mathcal{H}(G)$ such that $f(z)=F(z,\overline{z})$. Suppose that  $\gamma \in \Aut(\H)$. Then 
 $f$ is $\gamma$--invariant if and only if 
 $F$ is $\hat{\gamma}$--invariant.
  \end{lemma}

  \begin{proof}
    This is immediate from $$(f \circ \gamma)(z)=F\left(\gamma(z),\overline{\gamma(z)}\right)=F(\gamma(z),\gamma(\overline{z}))=(F \circ \hat{\gamma})(z,\overline{z})$$ which is based on the crucial property that every
    $\gamma \in \Aut(\H)$ not only maps $\H$ onto $\H$, but is symmetric w.r.t.~the real axis, that is, $\overline{\gamma(z)}=\gamma(\overline{z})$.
  \end{proof}

  We are therefore led to the following problem.

\begin{problem} \label{prob:1}
  Let $\Gamma$ be a  Fuchsian group acting on the upper half--plane $\H$.
%  \begin{itemize}
%  \item[(a)]
  Find the set of all $F \in \mathcal{H}(G)$ which are invariant under the subgroup
    $$ \hat{\Gamma}:=\{\hat{\gamma} \, : \, \gamma \in \Gamma\}$$
    of $\Aut(G)$.%, and 
%    \item[(b)] determine the quotient $G/\hat{\Gamma}$.
%    \end{itemize}
\end{problem}

It turns out that the so--called limit set of the Fuchsian group $\Gamma$ plays
the crucial role here.
 The limit set of $\Gamma$, denoted by $\Lambda(\Gamma)$, is the set of
accumulation points of the $\Gamma$--orbits of any point $z \in \H$. Since the
action of $\Gamma$ is properly discontinuous, $\Lambda(\Gamma) \subseteq
\partial \H:=\R \cup \{ \infty\}$.
 A~Fuchsian group $G$ is said to be elementary if $\Lambda(\Gamma)$  consists of exactly zero, one or  two elements.
Otherwise, it is called non--elementary. It is well--known (see
e.g.~\cite[Corollary 3.4.6]{Hub}) that the limit set of any non--elementary
Fuchsian group is either $\partial \H$ or a Cantor set.

    \begin{theorem} \label{thm:gen}
      Let $\Gamma$ be a non--elementary Fuchsian group. Then
every $\hat{\Gamma}$--invariant holomorphic function on $G$ is constant.
        \end{theorem}

 \begin{corollary} \label{cor:intro1}
  Suppose that $X$ is a hyperbolic Riemann surface
  with  a non--elementary covering group.   Then $\mathcal{A}(X)$ consists only of constant functions.
  \end{corollary}

  Corollary \ref{cor:intro1} applies in particular to all  compact Riemann surfaces of higher genus.
  We wish to point out that this strong \textit{rigidity result} only depends on the
  function--theoretic description
  of $\A$ in terms of $\mathcal{H}(\Omega)$ and conformal invariance.

  \medskip
  
  We now turn to the remaining cases, when $X$ is not simply connected and the Fuchsian group $\Gamma$ is
  elementary. Since we are only interested in hyperbolic Riemann surfaces, 
  we may restrict ourselves to the case that $\Gamma$ is a covering group, so it  acts freely on
  $\H$. Recalling that two  fixed--point free Fuchsian
  groups $\Gamma$ and $\Gamma'$ are conjugated if and only if the
  corresponding Riemann surfaces $\D/\Gamma$ and $\D/\Gamma'$ are conformally
  equivalent, the  well--known classification of elementary
    Fuchsian groups  (see \cite[Proposition 3.1.2]{Hub}) implies that  there are only two cases to be considered. Either $\Gamma$ is conjugate to a group generated by exactly one parabolic element (one fixed point on $\partial \H$) or 
by exactly one hyperbolic element (two fixed points on $\partial \H$), that is, $\D/\Gamma$ is either the punctured unit disk or a proper annulus.

\medskip

We first consider the case of a proper annulus $A_R=\{z \in \C \, : \, 1/R<|z|<R\}$ for some $R>1$.
Here $A_R=\H/\Gamma$, where $\Gamma$ is the cyclic group generated by $z \mapsto c\cdot  z$ with
$\log c=\pi^2/\log R>0$.

\begin{theorem}[Annulus] \label{thm:annulus}
  Let $\Gamma=\{z \mapsto c^n z \, : \, n \in \mathbb{Z}\}$ for some
  $c>1$.
Then $F\in \mathcal{H}(G)$   is $\hat{\Gamma}$--invariant if and only if 
    $$ F(z,w)=g\left( \frac{w}{z-w} \right)$$
    for some $g \in \mathcal{H}(\C)$.
    %{\color{red}Seb.:Note that on $\{(a,b,c):b^2-4ac=1\}$, these functions are simply  all $F(a,b,c)$ depending only on $b$.}
  \end{theorem}
On the Danielewski surface $D$ (see (\ref{eq:Dani})) the functions described in Theorem \ref{thm:annulus} are exactly all holomorphic functions $(a,b,c) \mapsto F(a,b,c)$ which depend only on $b$.

\medskip

  We finally  consider the case of the punctured unit disk $\D^*:=\D \setminus\{0\}=\H/\Gamma$, where
$\Gamma$ is the cyclic group generated by right--translation $z \mapsto z+1$ and associated universal covering
$z \mapsto e^{2\pi i z}$.

\begin{theorem} \label{thm:punctureddisk}
  Let $\Gamma:=\{z \mapsto z+n \, : \, n \in \mathbb{Z} \}$. Then
 $F \in \mathcal{H}(G)$ is $\hat{\Gamma}$--invariant if and only if 
    $$ F(z,w)=g\left( \frac{1}{z-w} \right)$$
    for some $g \in \mathcal{H}(\C)$. 
    %{\color{red}Seb.: Note that on $\{(a,b,c):b^2-4ac=1\}$, these functions are simply   all $F(a,b,c)$ depending only on $a$.}
    
  \end{theorem}

  On the Danielewski surface $D$ (see (\ref{eq:Dani})) the functions described in Theorem \ref{thm:punctureddisk} are exactly all holomorphic functions $(a,b,c) \mapsto F(a,b,c)$ which depend only on $a$.
  
 \begin{remark}[Elementary Fuchsian groups that do not act freely on $\H$]
   For completeness we describe all $\hat{\Gamma}$--invariant functions in $\mathcal{H}(G)$ for the case
   that $\Gamma$ is an elementary Fuchsian group which does not act freely on $\H$. There are basically two cases:

\medskip
Case 1: $\Gamma$ is a Fuchsian group generated by an elliptic element. On $\D$, such a group is conjugated to 
\[\Gamma' = \{z\mapsto \alpha^k\cdot z \, : \, k \in \mathbb{Z}\}, \quad \alpha=\exp(2\pi i /N), \, N\geq 2.\]
 We see that the quotient $\D/\Gamma'$ is isomorphic to $\D$. (The function $[z]\in \D/\Gamma' \mapsto z^N$ is a chart.) 
 The corresponding invariant functions on $\Omega$ satisfy $F(\alpha z, \overline{\alpha} w)=F(z,w)$. 
Each of the  functions $$f_{p,q}(z,w):=\frac{z^p w^q}{(1-zw)^{\max\{p,q\}}} \, , \quad  p,q\in\N_0 \, ,$$ belongs to $\mathcal{H}(\Omega)$ and conversely, $\mathcal{H}(\Omega)$ is the closure of the span of all $f_{p,q}$,
\cite[Theorem 3.16]{KRSW}. We conclude that the space of all invariant $F\in \mathcal{H}(\Omega)$ is the closure of the span of those $f_{p,q}$ with $(p-q)/N\in\mathbb{Z}$.

\medskip

The set of all functions $\D\ni z\mapsto F(z,\overline{z})$, $F\in \mathcal{H}(\Omega)$ $\hat{\Gamma}$--invariant,  can be represented on $\D$ by using the chart $z^N$. It consists of all functions 
\[g:\D\to\C\quad \text{with}\quad g(z^N)=F(z,\overline{z}), \quad F\in \mathcal{H}(\Omega).\]

Note that on $\D/\Gamma'$ this set  contains non-differentiable functions (non-differentiable at $[0]$). For example, let $N=2$ and consider the function $g(z)=\frac{|z|}{1-|z|}$ on $\D$. Then 
$g(z^2)=f_{1,1}(z,\overline{z})$ with $f_{1,1} \in \mathcal{H}(\Omega)$.\\

%Question: Is $\mathcal{A}(\D/\Gamma')$ isomorphic to $\mathcal{A}(\D)$ for all $N\geq2$?\\

Case 2: The Fuchsian group is conjugated to the group generated by $z\mapsto cz$ for some $c>1$, and $z\mapsto -1/z$ on $\H$.
This group is the direct product of the groups generated by $z\mapsto cz$ and $z\mapsto -1/z$. Thus we can take the quotient of $\H/(z\sim cz) =A_{R}=\{z\in \C \,:\, 1/R<|z|<R\}$ with respect to the involution $J(z)= 1/z$. 
The set of all functions $g \in \mathcal{A}(A_R)$ which are invariant under $z\mapsto 1/z$ can also be written as 
\[\left\{ z \mapsto g\left(\left(\tan\left(\frac{\pi}{2 \log
              R} \log |z|\right)\right)^2 \right)  \, : \, g \in \mathcal{H}(\C)  \right\}.\]
This follows immediately from Theorem \ref{thm:mainannulus}.
\end{remark}

\hide{We finally turn to Problem \ref{prob:1} (b) and investigate the quotient $G/\hat{\Gamma}$ for  a Fuchsian
group $\Gamma$.
 
 \begin{theorem} \label{thm:factor}
   Let $\Gamma$ be a Fuchsian group. Then the following are pairwise equivalent.
   \begin{itemize}
   \item[(a)] $\hat{\Gamma}$ acts freely on $G$;
   \item[(b)] $\hat{\Gamma}$ acts properly discontinuously on $G$;
     \item[(c)] $\Gamma$ is conjugate to $\{z \mapsto z+n \, : \, n \in \mathbb{Z}\}$.
     \end{itemize}
   \end{theorem}

   \begin{theorem} \label{thm:factor2}
     Let $\Gamma$ be a non--trivial Fuchsian group. Then the following are pairwise equivalent.
     \begin{itemize}
       \item[(a)] $G/ \hat{\Gamma}$ is a manifold;
       \item[(b)] $\H/\Gamma$ is conformally equivalent to the punctured disk $\D^*$; 
       \item[(c)] $\Gamma$ is conjugate to $\{z \mapsto z+n \, : \, n \in \mathbb{Z}\}$
\end{itemize}
       \end{theorem}

       \begin{theorem}[Annulus] \label{thm:factor3}
Let $\Gamma=\{z \mapsto c^n z \, : \, n \in \mathbb{Z}\}$ for some $c>1$ and
let $G':=G \setminus \{(\infty,0),(0,\infty)\}$. 
Then $G'/\hat{\Gamma}$ is a non--Hausdorff complex manifold.
\end{theorem}}

\section{M\"obius--invariant holomorphic functions on the second configuration space of the sphere -- Proof of Theorem \ref{thm:gen} and Theorem \ref{thm:annulus}} \label{sec:proofsMoebius}

For the proof of Theorem \ref{thm:gen} we recall that $\gamma \in
\Aut(\H)$ is said to be hyperbolic if it has two distinct fixed--points on
$\partial\H$, one repulsive and one attractive. If $\Gamma$ is a
non--elementary Fuchsian group, then its limit set $\Lambda(\Gamma)$ has the
following two well--known properties
\begin{itemize}
\item[(P1)] $\Lambda(\Gamma)$
is the
closure of the set of fixed points of the hyperbolic elements in $\Gamma$
(\cite[Thm.~3.4.4]{Katok});
\item[(P2)]
  $\Lambda(\Gamma)$ has no isolated points (\cite[Thm.~3.4.6]{Katok}).
  \end{itemize}
The idea of the proof of Theorem \ref{thm:gen} is as follows. In  a first
step, one proves that for  each fixed point $w_0$ of any hyperbolic element in
$\Gamma$ the holomorphic functions $F(w_0, \cdot)$ and $F(\cdot,w_0)$ are constant.
In a second step, which is based on (P1), it is then shown that $F$ is constant  on
$(\Lambda(\Gamma) \times \Lambda(\Gamma)) \cap G$. From this and  using (P2), one
finally can prove that $F$ is constant on~$G$.

 \begin{proof}[Proof of Theorem \ref{thm:gen}]
   Let $F \in \mathcal{H}(G)$ be $\hat{\Gamma}$-invariant.

   \smallskip
   
   (i) Let $\gamma$ be a hyperbolic automorphism in $\Gamma$ with fixed point $w_0
   \in \partial \H=\R\cup \{\infty\}$. We show that
   $F(\cdot,w_0) \in \H(\hat{\C} \setminus \{w_0\})$  and  $F(w_0,\cdot) \in
   \H(\hat{\C} \setminus \{w_0\})$ are constant.

   \smallskip
   Since $\gamma$ is hyperbolic, it  has a second fixed point $z_0 \in
   \partial \H \setminus \{w_0\}$. We may assume that $z_0 \in \R$, so 
$\gamma'(z_0) \not \in \partial \D$. The function
$g:=F(\cdot,w_0) \in \mathcal{H}(\hat{\C} \setminus \{w_0\})$ is
$\gamma$--invariant, since
\begin{equation} \label{eq:23}
 (g \circ \gamma)(z)=F(\gamma(z),w_0)=F(\gamma(z),\gamma(w_0))=F(z,w_0)=g(z)
 \, .
 \end{equation}
In particular, $g$ has a fixed point at $z_0$, and taking the derivative in
(\ref{eq:23}) at $z_0$ leads to $g'(z_0) \cdot \gamma'(z_0)=g'(z_0)$. Since
$\gamma'(z_0)\not=1$, this implies $g'(z_0)=0$. Taking the second derivative
in (\ref{eq:23}) at $z_0$ leads to $g''(z_0)=0$ and continuing this way, we
easily see that $g$ is constant. In a similar way, one can show that
$F(w_0,\cdot) \in \mathcal{H}(\hat{\C} \setminus \{w_0\})$ is constant.

\smallskip

(ii) Fix two distinct 
fixed points $p,q \in \partial \H$ of some hyperbolic automorphism in $\Gamma$ and let
$C:=F(p,q)$. Now, let $r,s \in \partial \H \setminus\{p,q\}$ be two distinct
fixed points of another  hyperbolic automorphism in $\Gamma$.
By (i),  $F(p,\hat{\C} \setminus\{p\})
=\{C\}$, so $F(p,s)=C$. Again by (i), we get  $F(\hat{\C} \setminus
\{s\},s)=C$, so $F(r,s)=C$. Since $\Gamma$ is non--elementary, the fixed points of all hyperbolic
elements in $\Gamma$ are dense in $\Lambda(\Gamma)$. This implies that $F$ is constant $=C$ on
$(\Lambda(\Gamma) \times \Lambda(\Gamma)) \cap G$ as follows. Let $(p,q) \in \left(\Lambda(\Gamma) \times \Lambda(\Gamma) \right) \cap \Omega$. In particular, $p\not=q$. Then there are sequences $(p_n)$, $(q_n) \subseteq \partial \H$, $(T_n) \subseteq \Gamma$ such that each $T_n$ is hyperbolic  and $T_n(p_n)=p_n \to p$ as well as $T_n(q_n)=q_n\to q$ as $n \to \infty$. Hence
$q_n \not=p_n$ for $n$ sufficiently large, so $F(p_n,\cdot) \equiv C$ on $\hat{\C} \setminus \{p_n\}$ and thus $F(p,q)=\lim_{n \to \infty} F(p_n,q_n)=C$.
\smallskip

(iii) Since $\Lambda(\Gamma)$ contains no isolated points, it is clear that any
 $h \in \mathcal{H}(G)$ which is constant on $(\Lambda(\Gamma) \times
\Lambda(\Gamma)) \cap G$ has the property that ${\partial h/\partial z}={\partial
  h/\partial w}=0$ on $(\Lambda(\Gamma) \times
\Lambda(\Gamma)) \cap G$.
In particular,  this implies that for any $(p,q) \in (\Lambda(\Gamma) \times
\Lambda(\Gamma)) \cap G$ all partial derivatives of $F$ 
at the point $(p,q)$ vanish. Hence $F$ is constant.
\end{proof}

\begin{proof}[Proof of Theorem \ref{thm:annulus}]
Since $$\hat{\Gamma}=\{ (z,w) \mapsto (c^n z,c^n w) \, : \, n \in \mathbb{Z}\}\,
, $$
it is clear that for every entire function $g : \C \to \C$ the holomorphic function
$$ F : G \to \C \, , \qquad F(z,w):=g \left(\frac{w}{z-w} \right) \, , $$
is $\hat{\Gamma}$--invariant. In order to prove the only--if part of Theorem
\ref{thm:annulus}, let $F \in \mathcal{H}(G)$ be $\hat{\Gamma}$--invariant.
We define
$$ H \in \mathcal{H}\left(\C^2 \setminus \left(\{ (1,v) \, : \, v \in \C \} \cup \{(u,0) \, : \, u
  \in \C\} \right)\right) \, , \qquad H(u,v):=F(v/u,v) \, .$$
For fixed $u \in \C \setminus \{1\}$, we have
$$H(u,cv)=F(cv/u,cv)=F(v/u,v)=H(u,v) \, , \qquad v \in \C \setminus \{0\} \,
.$$
Comparing the coefficients in the Laurent expansions about the point $v=0$ on
both sides, we see that
 $H(u,\cdot)$ is constant $=H(u,1)$ on $\C \setminus \{0\}$. 
This implies
that
$$ F(z,w)=H(w/z,w)=H(w/z,1) \qquad \text{ for all }  (z,w) \in G \text{ s.t. } z\not=0 \, .$$
We now set
$$  g(s):=H\left(\frac{s}{s-1},1\right) \,  \qquad s \in  \C
\setminus\{1\} \,.$$
Then $g \in \mathcal{H}(\C\setminus \{1\})$ has a holomorphic extension to
$\C$, since
$$ g(s)=
F\left(1-\frac{1}{s},1\right) \, , \quad s \in \C \setminus \{1\}\,, $$
and $F$ is holomorphic on $G$. This shows that
$$ F(z,w)=g\left( \frac{1}{1-z/w} \right)$$
for all $(z,w) \in G$.
  \end{proof}

 \begin{proof}[Proof of Theorem \ref{thm:punctureddisk} (a)]
Since $$\hat{\Gamma}=\{ (z,w) \mapsto (z+n,w+n) \, : \, n \in \mathbb{Z}\}\,
, $$
it is clear that for every entire function $g : \C \to \C$ the function
$$ F : G \to \C \, , \qquad F(z,w):=g \left(\frac{1}{z-w} \right) \, , $$
is $\hat{\Gamma}$--invariant. In order to prove the only--if part of Theorem
\ref{thm:punctureddisk}, let $F \in \mathcal{H}(G)$ be a
$\hat{\Gamma}$--invariant function. We proceed in two steps.

\medskip

(i) We first note  that for any $k=0,1,2\ldots$ the
entire function
$$ \frac{\partial^k F}{\partial
                  w^k}(\cdot,\infty)\, $$
 is constant. In fact,
$$ \frac{\partial^k F}{\partial
  w^k} \in \mathcal{H}(\Omega)$$
is $\hat{\Gamma}$--invariant, so that for fixed $z,w \in \C$, we have
\begin{eqnarray*}
\frac{\partial^k F}{\partial
                  w^k}(z,\infty) &=& \lim \limits_{n \to \infty} \frac{\partial^k F}{\partial
                  w^k}(z,w+n)=\lim \limits_{n \to
                \infty} \frac{\partial^k F}{\partial
                  w^k}(z-n+n,w+n)\\[2mm] &=&\lim \limits_{n \to \infty}
                \frac{\partial^k F}{\partial
                  w^k}(z-n,w)=\frac{\partial^k F}{\partial
                  w^k}(\infty, w) \, .
\end{eqnarray*}
This proves (i).

\smallskip
(ii) Define
$$ H \in \mathcal{H}(\C \times \hat{\C} \setminus \{0\})\, , \qquad
H(u,v):=F(u,u-v) \, .$$
Since $F$ is $\hat{\Gamma}$--invariant, we see that for every fixed $v \in
\hat{\C} \setminus \{0\}$ the function $u \mapsto H(u,v)$ is a $1$--periodic
entire function and therefore can be written as
$$ H(u,v)=\sum \limits_{n=-\infty}^{\infty} a_n(v) e^{2 \pi i n u} $$
with ``Fourier coefficients'' $a_n \in \mathcal{H}(\hat{\C} \setminus \{0\})$
for $n =0,\pm 1, \pm 2, \ldots$.
Now (i) implies that for each $k=0,\pm 1,\pm 2, \ldots$
the $1$--periodic entire function
$$ \frac{\partial^k H}{\partial v^k}(\cdot,\infty)$$
is constant and hence
$$ \frac{d^k a_n}{d v^k}(\infty)=0 \quad \text{ for all } n=\pm 1, \pm 2,
\ldots \, .$$
For any $n=\pm 1, \pm 2,
\ldots$ the holomorphic function  $a_n : \hat{\C} \setminus \{0\} \to \C$  therefore
vanishes identically  and we conclude
$$ H(u,v)=a_0(v) \, , \qquad (u,v) \in \C \times \hat{\C} \setminus \{0\}
\,. $$
This shows that
$$F(z,w)=H(z,z-w)=a_0(z-w) \text{ for all } (z,w) \in G \, $$
with a function $a_0 \in \mathcal{H}(\hat{\C} \setminus \{0\})$. The proof of
Theorem \ref{thm:punctureddisk} is complete.
\end{proof}

\hide{\begin{proof}[Proof of Theorem \ref{thm:factor}]
(a) $\Longrightarrow$ (c):
  Let $\hat{\Gamma}$ act freely on $\Omega$. Then every  element in $\Gamma
  \setminus\{\text{id}\}$ has only one fixed point, which then lies on
  $\partial \H$. In particular, $\Gamma$ is an elementary Fuchsian group by
  \cite[Thm.~2.4.4]{Katok}.
  By \cite[Thm. 2.4.3 (proof)]{Katok}, $\Gamma$ is cyclic, so we may assume that
  $\Gamma = \{z \mapsto z+n \, :
    \, n \in \mathbb{Z}\}$.

    \medskip

    (b) $\Longrightarrow$ (c):
    Let $\hat{\Gamma}$ act discontinuously on $\Omega$. It suffices to show
    that $\Gamma$ contains only parabolic elements.
    Otherwise, $\Gamma$ contains a non--parabolic element $\gamma$. By
    conjugation, we may assume that $\gamma(z)=cz$ with $0<c<1$. Consider the
    compact set $K=\{(1/2,z) \, : z \in \hat{\C} \setminus \D\} \subseteq \Omega$.
    Then $\hat{\gamma}^n(K)=\{(c^n/2, \hat{\C} \setminus \D_{c^n})\}$ has the
    limit point $(0,1) \in \Omega$, say. Hence $\hat{\Gamma}$ does not act
    properly discontinuous on $\Omega$, a contradiction.
      \end{proof}}

\hide{
\begin{proof}[Proof of Theorem \ref{thm:factor2}]
{\color{red}{ToDo}}
\end{proof}  
\begin{proof}[Proof of Theorem \ref{thm:factor3}]
{\color{red}{ToDo}}
\end{proof}  }

\section{The convergent Wick--star product on doubly connected Riemann surfaces} \label{sec:Doublyconnected}

In this section we prove Theorem \ref{thm:mainannulus} and  Corollary \ref{cor:mainannulus} as well as Theorem \ref{thm:mainpunctureddisk}  and Theorem \ref{cor:mainpunctureddisk}. The main tools will be (i) an explicit formula for the star product $\star_{\hbar,\D}$ on the unit disk $\D$ and (ii) an asymptotic expansion for it as $\hbar \to 0$. Both results have recently been established in \cite{HeinsMouchaRoth1} and rely essentially on formula (5.24) in \cite{SchmittSchoetz2022}. The explicit formula for $\star_{\D}$ is based
on invariant derivatives of Peschl--Minda type, which are defined as follows. For $z \in \D$ denote
$$ T_z(u):=\frac{z+u}{1+\overline{z} u} \in \Aut(\D)\, ,$$
and by
$$\partial:=\frac{1}{2} \left( \frac{\partial}{\partial x}-i \frac{\partial}{\partial y}\right)$$  the standard (Wirtinger) differential operator.
Then for every $C^{\infty}$--function $f : \D \to \C$ the Peschl--Minda derivatives of $f$  of order $n$ at the point $z \in D$ are defined by 
$$ D^nf(z):=\partial^n (f \circ T_z)(0) \, , \qquad \overline{D}^nf(z):=\overline{D^n(\overline{f})(z)}\, ,$$
see e.g.~\cite{HeinsMouchaRoth1,KS07diff}. Corollary 5.1 in \cite{HeinsMouchaRoth1} provides the following explicit formula for $D^nf$ in terms of the Wirtinger derivative $\partial$:
\begin{equation} \label{eq:PM2}
D^{n+1} f(z)= \left( 1-|z|^2 \right)
                \partial^{n+1}
                \left[ \left(1-|z|^2 \right)^n f(z) \right]
                \end{equation}
Unlike $\partial^{n+1}=\partial^n \partial$ one has $D^{n+1} \not=D^n D^1$ in general. The basic result we need is now

\begin{equation} \label{eq:*D}
(f \star_{\D,\hbar} \tilde{f})(z)=\sum \limits_{n=0}^{\infty} \frac{c_n(\hbar)}{n!} \,  D^n\tilde{f}(z) \cdot \overline{D}^nf(z) \, 
\end{equation}
 for every $f,\tilde{f} \in \A$ and every $\hbar \in \mathscr{D}$, and with $c_n(\hbar)$ defined as in (\ref{eq:coeff}).
 In fact, for each fixed $z \in \D$ the  series (\ref{eq:*D}) converges  absolutely and locally uniformly w.r.t.~$\hbar \in \mathscr{D}$, so $\hbar \mapsto (f \star_{\D,\hbar} \tilde{f})(z)$ is holomorphic on $\mathscr{D}$. Moreover, for fixed $\hbar \in \mathscr{D}$, the series (\ref{eq:PM2}) converges in the topology of $\mathcal{A}(\D)$. See \cite[Lemma 6.5]{HeinsMouchaRoth1}.

\begin{proof}[Proof of Theorem \ref{thm:mainannulus}] The proof is divided into several steps.
  
(i)   
We first prove that the linear mapping $T_R : \mathcal{H}(\C) \to \mathcal{A}(A_R)$, $T_R(g)=g \circ f_R$, is bijective and establish (\ref{eq:A(A_R)}). For this purpose we work on the upper half plane $\H$. We note that 
      $$\pi : \H \to A_R \, , \qquad \pi(z)=\exp \left( \frac{ 2i\log R}{\pi} \log
        \left( \frac{z}{i} \right) \right)\, , $$
      is a universal covering of $A_R$ and $\Gamma=\{z \mapsto c^n z \, :
  \, n \in \mathbb{Z}\}$ is the   associated group of deck transformations,
  where
  $$ \log c=\frac{\pi^2}{\log R} \, ,$$
  and $\log$ denotes the principal branch of the logarithm.
Let  $f : A_R \to \C$ be a function. Then $f$ belongs to $\mathcal{A}(A_R)$ if and only
if  the $\Gamma$--invariant function $f \circ \pi$ belongs to $\mathcal{A}(\H)$. By
Lemma \ref{lem:2_1} this is equivalent to  $(f \circ
\pi)(z)=F(z,\overline{z})$ for all $z \in \H$ for some uniquely determined $\hat{\Gamma}$--invariant function $F \in
\mathcal{H}(G)$. Theorem \ref{thm:annulus} therefore shows that $f \in
\mathcal{A}(A_R)$ if and only if  
$$ (f \circ \pi)(z)=\tilde{g} \left( \frac{\overline{z}}{z-\overline{z}}
  \right) \, , \qquad z \in \H \, , $$
  for some  uniquely determined function $\tilde{g} \in \mathcal{H}(\C)$. In fact, the proof of Theorem \ref{thm:annulus} reveals that  $\tilde{g}(s)=F(1-1/s,1)$ for all $s \in \C \setminus \{1\}$.
Also note that
$$ \frac{\overline{z}}{z-\overline{z}}=-\frac{i}{2} \tan \left( \frac{\pi}{2
      \log R} \log |\pi(z)| \right)-\frac{1}{2}=\frac{1}{2} f_R(\pi(z))-\frac{1}{2} \, , \qquad z \in \H \, , $$
with $f_R$ as defined in (\ref{eq:f_RDef}). Hence, if $f \in \mathcal{A}(A_R)$, then 
 $ f=g  \circ f_R$ for the entire function $g(z):=\tilde{g}(z/2-1/2)$, and  conversely, $f:=g \circ f_R$ belongs to $\mathcal{A}(A_R)$ for every $g \in \mathcal{H}(\C)$. This proves (\ref{eq:A(A_R)}), and  also that
 the  mapping $T_R : \mathcal{H}(\C) \to \mathcal{A}(A_R)$, $T_R(g):=g \circ f_R$, is a linear bijection.

(ii) As a next step, we prove that  $T_R^{-1} : \mathcal{A}(A_R) \to \mathcal{H}(\C)$ is continuous. Suppose $(f_n) \subseteq \mathcal{A}(A_R)$ converges in $\mathcal{A}(A_R)$ to $f \in \mathcal{A}(A_R)$. Then  $(f_n\circ \pi)(z)=F_n(z,\overline{z})$ and $(f \circ \pi)(z)=F(z,\overline{z})$ for uniquely determined $\hat{\Gamma}$--invariant functions $F_n,F \in \mathcal{H}(G)$ and $f_n \to f$ in $\mathcal{A}(A_R)$ is the same  as $F_n \to F$ in $\mathcal{H}(G)$ by definition. As shown in (i), there are entire functions $\tilde{g}_n$, $\tilde{g}$ such that $\tilde{g}_n(s)=F_n(1-1/s,1)$ and $ \tilde{g}(s)=F(1-1/s,1)$ for all $s \in \C \setminus \{1\}$. It follows that $\tilde{g}_n \to \tilde{g}$  locally uniformly on $\C \setminus \{1\}$, which implies locally uniform convergence on $\C$ by the maximum principle. Hence $g_n(z):=\tilde{g}_n(z/2-1/2)$ and $g(z):=\tilde{g}(z/2-1/2)$ are entire functions such that $g_n \to g$ in $\mathcal{H}(\C)$ and $f_n=g_n \circ f_R=T_R(g_n)$ as well as $f=g \circ f_R=T_R(g)$. This means that $T_R^{-1}(f_n)=g_n \to g=T_R^{-1}(f)$ in $\mathcal{H}(\C)$, so $T_R^{-1}: \mathcal{A}(A_R) \to \mathcal{H}(\C)$ is continuous, as required. By the bounded inverse theorem for Fr\'echet spaces, $T_R : \mathcal{H}(\C) \to \mathcal{A}(A_R)$ is continuous as well.

  \medskip
  
(iii)   In order to prove the explicit formula  (\ref{eq:FormulaStarAnnulus}) for the star product $\star_{A_R,\hbar}$ for the annulus $A_R$ we work on the unit disk $\D$ and note that
  $$ \pi_R : \D \to A_R \, , \qquad \pi_R(z):=\exp \left( \frac{2i}{\pi} \log \frac{1+z}{1-z} \right)$$
  is a universal covering of $A_R$. Consider the auxiliary function
  $$ p : \D \to \C \, , \qquad p(z):=\frac{z-\overline{z}}{1-|z|^2}\, , $$ and recall the  definition of the function $f_R : A_R \to \C$ in (\ref{eq:f_RDef}). Then
  $$ f_R \circ \pi_R =p \, .$$
  A straightforward proof by induction and making use of (\ref{eq:PM2}) reveals that
  \begin{eqnarray*}
    D^n(g \circ p)(z)& =&\left( \frac{1-\overline{z}^2}{1-|z|^2}\right)^n \, \left( g^{(n)} \circ p \right)(z) \\
    \overline{D}^n(g \circ p)(z) & =&(-1)^n \left( \frac{1-z^2}{1-|z|^2}\right)^n \, \left( g^{(n)} \circ p \right)(z)\, ,
    \end{eqnarray*}
    and hence
    $$ \overline{D}^n (g \circ p) D^n(\tilde{g} \circ p)=\left( p^2-1 \right)^n \left( g^{(n)} \circ p \right) \left( \tilde{g}^{(n)} \circ p \right) $$
    for all $g, \tilde{g} \in \mathcal{H}(\C)$. Inserting this formula with $f=g \circ f_R$ and $\tilde{f}=\tilde{g} \circ f_R$ into (\ref{eq:*D}) establishes (\ref{eq:FormulaStarAnnulus}). That for fixed $\hbar \in \mathscr{D}$, the series (\ref{eq:FormulaStarAnnulus}) converges in the topology of $\mathcal{A}(A_R)$ is a direct consequence of the fact that the series (\ref{eq:*D}) converges in the topology of $\mathcal{A}(\D)$, see \cite[Lemma 6.5]{HeinsMouchaRoth1}. Alternatively, it is not difficult to show directly that the series
    $$ \sum \limits_{n=0}^{\infty}  \frac{c_n(\hbar)}{n!} \left(w^2-1\right)^n g^{(n)}(w) \tilde{g}^{(n)}(w)$$
converges locally uniformly w.r.t.~$w \in \C$ for every fixed $\hbar \in \mathscr{D}$ by making appeal to the Cauchy estimates for the entire functions $g$ and $\tilde{g}$. 
  \end{proof}

\begin{proof}[Proof of Corollary \ref{cor:mainannulus}]
  By Theorem \ref{thm:mainannulus}, the mapping $\Psi_{R',R}=T_R \circ T_{R'}^{-1} : \mathcal{A}(A_{R'}) \to \mathcal{A}(A_{R})$ is a continuous linear bijection. 
  It remains to prove that $\Psi_{R',R}$ is compatible with the algebra structures of $(\mathcal{A}(A_R),\star_{A_R})$ and $(\mathcal{A}(A_{R'}),\star_{A_{R'}})$.  First note that the definition of $\Psi_{R',R}$ implies
\begin{equation} \label{eq:Psi}
  \Psi_{R',R} (h \circ f_{R'})=h \circ f_R \qquad \text{ for all } h \in \mathcal{H}(\C) \, .
  \end{equation}
    Let $f, \tilde{f} \in \mathcal{A}(A_{R'})$ and $g, \tilde{g} \in \mathcal{H}(\C)$ such that $f=g \circ f_{R'}=T_{R'}(g)$ and $\tilde{f}=\tilde{g} \circ f_{R'}=T_{R'}(\tilde{g})$.
    Applying (\ref{eq:Psi}) for the entire function $h(w):=(w^2-1)^n g^{(n)}(w) \tilde{g}^{(n)}(w)$ shows that
    \begin{equation} \label{eq:Psi2}
    \Psi_{R',R}\left[ \left(f_{R'}^2-1 \right)^n \left( g^{(n)} \circ f_{R'}\right) \left(  \tilde{g}^{(n)} \circ f_{R'}\right)\right] =\left( f_{R}^2-1 \right)^n \left( g^{(n)} \circ f_R \right) \left( \tilde{g}^{(n)} \circ f_R \right) \, .
    \end{equation}

    Since the series (\ref{eq:FormulaStarAnnulus}) converges for fixed $\hbar \in \mathscr{D}$ in the topology of $\mathcal{A}(A_R)$ and $\Psi_{R',R} : \mathcal{A}(A_{R'}) \to \mathcal{A}(A_R)$ is continuous, we deduce 
   \begin{eqnarray*}
      \Psi_{R',R}(f) \star_{A_{R}} \Psi_{R,',R}(\tilde{f}) &=& \Psi_{R',R}(g \circ f_{R'}) \star_{A_{R}} \Psi_{R',R}( \tilde{g} \circ f_{R'}) \\
                                                           &\overset{(\ref{eq:Psi2})}{=}& (g \circ f_R) \star_{A_R} (\tilde{g} \circ f_R)\\
                                                           &\overset{(\ref{eq:FormulaStarAnnulus})}{=}& \sum \limits_{n=0}^{\infty} c_n(\hbar) \left( f_{R}^2-1 \right)^n \left( g^{(n)} \circ f_R \right) \left( \tilde{g}^{(n)} \circ f_R \right)\\
      &\overset{(\ref{eq:Psi2})}{=}&  \sum \limits_{n=0}^{\infty}   c_n(\hbar) \Psi_{R',R}\left[\left( f_{R'}^2-1 \right)^n \left( g^{(n)} \circ f_{R'} \right) \left( \tilde{g}^{(n)} \circ f_{R'} \right)\right] \\
                                                           &=& \Psi_{R',R} \left[ \sum \limits_{n=0}^{\infty} c_n(\hbar) \left( f_{R'}^2-1 \right)^n \left( g^{(n)} \circ f_{R'} \right) \left( \tilde{g}^{(n)} \circ f_{R'} \right) \right]\\ &\overset{(\ref{eq:FormulaStarAnnulus})}{=}& 
      \Psi_{R',R}\left( f \star_{A_R'} \tilde{f} \right) \, .
    \end{eqnarray*}
  \end{proof}
    
    \begin{proof}[Proof of Theorem \ref{thm:mainpunctureddisk}]
The proof is very similar to the proof of Theorem \ref{thm:mainannulus}, so we only indicate the main steps. 
We first prove (\ref{eq:A(D^*)}). To this end,  we note that 
      $\pi : \H \to \D^*$, $\pi(z)=e^{2\pi
    iz}$, is a universal covering of $\D^*$ and $\Gamma=\{z \mapsto z+n \, :
  \, n \in \mathbb{Z}\}$ is the associated group of deck transformations. 
Hence a function  $f : \D^* \to \C$ belongs to $\mathcal{A}(\D^*)$ if and only
if  the $\Gamma$--invariant function $f \circ \pi$ belongs to $\mathcal{A}(\H)$. By
Lemma \ref{lem:2_1} this is equivalent to  $(f \circ
\pi)(z)=F(z,\overline{z})$ for some $\hat{\Gamma}$--invariant function $F \in
\mathcal{H}(G)$. Theorem \ref{thm:punctureddisk} therefore shows that $f \in
\mathcal{A}(\D^*)$ if and only if  
$$ f(e^{2 \pi i z})=(f \circ \pi)(z)=\tilde{g} \left( \frac{1}{z-\overline{z}}
  \right) \, , \qquad z \in \H \, , $$
  for some entire function $\tilde{g} : \C \to \C$ or
  $$ f(w)=g \left( \frac{1}{\log |w|} \right)=g(f_0(w)) \, , \qquad w \in \D^* \, ,$$
  for the entire function $g(z):=\tilde{g}(-\pi z/i)$. Conversely, $f:= g \circ f_0$ belongs to $\mathcal{A}(\D^*)$ for every $g \in \mathcal{H}(\C)$, proving (\ref{eq:A(D^*)}).

  \medskip

  Finally, we prove the explicit formula (\ref{eq:FormulaStarPuncturedDisk}) for the punctured disk $\D^*$, and for this purpose we work on the unit disk $\D$ and employ the universal covering map $\pi_0 : \D \to \D^*$ given by
  $$ \pi_0(z):=\exp \left(-\frac{1+z}{1-z} \right) \, .$$
  Consider the auxiliary function
  $$ q : \D \to \C \, , \qquad q(z):=\frac{|1-z|^2}{1-|z|^2} $$
  and the function $f_0 : \D^* \to \C$ defined in (\ref{eq:f_0Def}). Then
  $$ \pi_0 \circ f_0=q \, , $$
  and it easily follows inductively from (\ref{eq:PM2}) that
    \begin{eqnarray*}
    D^n(g \circ q)(z)& =&(-1)^n \frac{(1-\overline{z})^{2n}}{(1-|z|^2)^n} \, \left( g^{(n)} \circ q \right)(z) \\
    \overline{D}^n(g \circ q)(z) & =&(-1)^n  \frac{(1-z)^{2n}}{(1-|z|^2)^n} \, \left( g^{(n)} \circ q \right)(z)\, 
    \end{eqnarray*}
    and hence
\begin{equation} \label{eq:Bi}
  \overline{D}^n (g \circ q) D^n(\tilde{g} \circ q)=q^{2n}  \left( g^{(n)} \circ q \right) \left( \tilde{g}^{(n)} \circ q \right)
  \end{equation}
     for all $g, \tilde{g} \in \mathcal{H}(\C)$. Plugging this expression into (\ref{eq:*D}) results in
     (\ref{eq:FormulaStarPuncturedDisk})
    \end{proof}

    \begin{proof}[Proof of Theorem \ref{cor:mainpunctureddisk}]
The proof is by contradiction. 
      Assume there is a strong Fr\'echet algebra isomorphism $\Psi : (\mathcal{A}(A_R),\star_{A_R}) \to (\mathcal{A}(\D^*),\star_{\D^*})$. 

      (i) Consider $f_R \in \mathcal{A}(A_R)$ as defined by (\ref{eq:f_RDef}). Then (\ref{eq:FormulaStarAnnulus}) for $g(z)=\tilde{g}(z)=z$
implies
$$ f_R \star_{A_R} f_R=f_R^2+\left( f_R^2-1 \right) \hbar \, $$
and hence
$$ \Psi(f_R \star_{A_R} f_R)=\Psi\left(f_R^2+\hbar\left(f_R^2-1\right) \right)=\Psi(f_R^2)+ \left(\Psi(f_R^2)-1\right) \hbar\, .$$
Here, we have used that $\Psi$ is a linear map and $\Psi(1)=1$, with $1$ denoting the function with constant value $1$.

\medskip
(ii) By Theorem \ref{thm:mainpunctureddisk}, there is a unique $g \in \mathcal{H}(\C)$ such that
$$ \pi_0 \circ \Psi(f_R) =g \circ q \, .$$
Here, $\pi_0$ and $q$ have the same meaning as in the proof of Theorem \ref{thm:mainpunctureddisk}.
It follows from the definition of the star product $\star_{\hbar,\D^*}$ that
$$\left(   \Psi(f_R) \star_{\hbar,\D^*} \Psi(f_R) \right) \circ \pi_0=( g \circ q) \star_{\hbar,\D} (g \circ q) \, .$$
We now make use of the asymptotic expansion for $\star_{\hbar,\D}$ as $\hbar \to 0$ established in \cite[Theorem 6.9]{HeinsMouchaRoth1}, which asserts
\begin{eqnarray*} 
  ( g \circ q) \star_{\hbar,\D} (g \circ q)&=& (g \circ q)^2+\hbar \overline{D}^1(g \circ q) D^1(g \circ q)+\frac{\hbar^2}{2} \overline{D}^2(g \circ q) D^2(g \circ q)+O(\hbar^3)\\
&\overset{(\ref{eq:Bi})}{=}& \Psi(f_R)^2 \circ \pi_0+\hbar q^2  (g' \circ q)^2+\frac{\hbar^2}{2} q^4 (g'' \circ q)^2+O(\hbar^3)
  \, .
                                               \end{eqnarray*}
(iii) Combing (i) and (ii), we see that under the assumption that $\Psi(f_R \star_{\hbar,A_R} f_R)=  \Psi(f_R) \star_{\hbar,\D^*} \Psi(f_R) $ for all $\hbar \in \mathscr{D}$,
$$ \Psi(f_R^2) \circ \pi_0+ \left(\Psi(f_R^2)\circ \pi_0-1\right) \hbar= \Psi(f_R)^2 \circ \pi_0+\hbar q^2  (g' \circ q)^2+\frac{\hbar^2}{2} q^4 (g'' \circ q)^2+O(\hbar^3) \, .
$$
Comparing equal powers of $\hbar$,
we get
\begin{itemize}
\item[(I)] $\hbar^2$: $\frac{1}{2} q^4 (g'' \circ q)^2=0$, so $g''=0$, and hence $g(z)=\alpha z+\beta$ for some $\alpha,\beta \in \C$.
\item[(II)] $\hbar^1$: $\Psi(f_R^2)\circ \pi_0-1=q^2 \alpha^2$.
\item[(III)] $\hbar^0$: $\Psi(f_R^2) \circ \pi_0=\Psi(f_R)^2 \circ \pi_0$ and hence $\Psi(f_R^2) \circ \pi_0=(g \circ q)^2=(\alpha^2 q^2+2 \alpha \beta q+\beta^2)$. 
\end{itemize}
Combining (II) and (III) gives $\alpha=0$ and $\beta=\pm 1$. Therefore, $\Psi(f_R)=\pm 1=\pm \Psi(1)=\Psi(\pm 1)$, and thus $f_R=\pm 1$, a contradiction.
%(v) It follows from (iv) 
\end{proof}

\section{Proof of Theorem \ref{thm:main}, Corollary \ref{cor:1} and Corollary \ref{cor:2}} \label{sec:proofs}

\begin{proof}[Proof of Theorem \ref{thm:main}]
The ``only if'' parts of (a), (b) and (c) follow from Proposition \ref{prop:confinv}.
In order to prove the ``if part'' of (a) let $\mathcal{A}(X) \si \mathcal{A}(\D)$. In particular, $\mathcal{A}(X)$ contains non--constant functions and $\star_X$ is noncommutative. Therefore Corollary \ref{cor:intro1} implies that  the covering group of $X$ is elementary, and Theorem \ref{thm:mainannulus} and Theorem \ref{thm:mainpunctureddisk} imply that $X$ is not  conformally equivalent to a proper annulus or the punctured disk. Hence $X$ is simply connected.
Now assume $\mathcal{A}(X) \si \mathcal{A}(A_R)$ for some $R>1$. Then, by Corollary \ref{cor:mainannulus}, $\mathcal{A}(X) \si \mathcal{A}(A_{R'})$ for every $R'>1$. Moreover, $(\mathcal{A}(X),\star_X)$ is not strongly isomorphic to $(\mathcal{A}(\D^*),\star_{\D^*})$ by Theorem \ref{cor:mainpunctureddisk} and hence $X$ is not conformally equivalent to $\D^*$ by Proposition \ref{prop:confinv}. $X$ cannot be conformally equivalent to $\D$ since otherwise $(\mathcal{A}(X),\star_X)$ would be noncommutative, contradicting Theorem \ref{thm:mainannulus}.  We conclude that $X$ has to be conformally equivalent to some proper annulus $A_\rho$, $\rho>1$. This proves the ``if part'' of (b). The ``if part'' of (c) follows in a similar way. To prove (d) note that if $X$ is neither conformally equivalent to $\D$ nor to some annulus or punctured disk, then the covering group of $X$ is non--elementary, so $\mathcal{A}(X)$ consist only of constant functions by Corollary \ref{cor:intro1}.
\end{proof}

\begin{proof}[Proof of Corollary \ref{cor:2}]
  Note that $\mathcal{A}(A_R)$ and $\mathcal{A}(\D^*)$ do not separate points since they consists only of radially symmetric functions by Theorem \ref{thm:mainannulus} and Theorem \ref{thm:mainpunctureddisk}. On the other hand,
  $\mathcal{A}(\D)$ separates points, since $\mathcal{H}(\Omega)$ separates points as $\Omega$ is a Stein manifold, see Theorem \ref{thm:G}. This proves that (a) and (b) are equivalent. The equivalence of (a) and (c) follows directly from Theorem \ref{thm:main}, Theorem \ref{thm:mainannulus} and Theorem \ref{thm:mainpunctureddisk}.
  \end{proof}

Daniela Kraus, Oliver Roth, Sebsatian Schlei{\ss}inger, Stefan Waldmann\\
Department of Mathematics\\
University of W\"urzburg\\
Emil Fischer Strasse 40\\
97074 W\"urzburg, Germany

\bigskip

dakraus@mathematik.uni-wuerzburg.de\\
roth@mathematik.uni-wuerzburg.de\\
stefan.waldmann@mathematik.uni-wuerzburg.de\\


\begin{thebibliography}{Car03}

\bibitem{AW} R. B. Andrist, E. F. Wold, Free dense subgroups of holomorphic automorphisms, \emph{Math. Z.} 280 (2015), no. 1--2, 335--346. 

\bibitem{5} A. Cano, J. P. Navarrete, J. Seade, \emph{Complex Kleinian Groups}, 
  Birkh\"auser Mathematics, 2013.

\bibitem{BW} S. Beiser, S. Waldmann, Fr\'echet algebraic deformation quantization of the Poincar\'e disk, \emph{J.~Reine Angew.~Math}.~\textbf{688} (2014) 147–-207.

  \bibitem{BDS} P.~Bieliavsky, S.~Detournay, SP.~Spindel,
The deformation quantizations of the hyperbolic plane, \emph{Commun. Math. Phys}.~\textbf{289}, No.~2 (2009), 529--559. 
  
\bibitem{CGR} M.~Cahen, S.~Gutt, J.~Rawnsley,   Quantization of K\"ahler manifolds. III, \emph{Lett. Math. Phys}.~\textbf{30}, No.~4 (1994) 291--305. 
  
\bibitem{CE} K. Cieliebak, Y. Eliashberg, \emph{From Stein to Weinstein and back. Symplectic geometry of affine complex manifolds.}
American Mathematical Society, Providence, RI, 2012.

\bibitem{Coh} F.R.~Cohen, \emph{Introduction to Configuration Spaces and
    their Applications}, In Braids,
volume 19 of Lect. Notes Ser. Inst. Math. Sci. Natl. Univ. Singap., pages 183–261. World
Sci. Publ., Hackensack, NJ, 2010.
\bibitem{HM} B. Hall, J. J. Mitchell, Coherent states on spheres, \emph{Journal of Mathematical Physics} 43 (2002), 1211--1236.
\bibitem{HM2}  B. Hall, J. J. Mitchell, Erratum: Coherent states on spheres [J. Math. Phys. 43, 1211 (2002)], \emph{Journal of Mathematical Physics} 46 (2005), 059901.

\bibitem{HeinsMouchaRoth1} M.~Heins, A.~Moucha, O.~Roth and T.~Sugawa, \textit{Peschl-Minda invariant differential operators and convergent Wick star products on the disk, the sphere and beyond}, see  \href{https://arxiv.org/abs/2308.01101}{arXiv:2308.01101}



\bibitem{HeinsMouchaRoth3} M.~Heins, A.~Moucha and O.~Roth, \textit{Function Theory off the complexified unit circle: Fr\'echet space structure and automorphisms},  see \href{https://arxiv.org/abs/2308.01107}{arXiv:2308.01107}


  
      \bibitem{HeinsMouchaRoth2} M.~Heins, A.~Moucha and O.~Roth, \textit{Spectral theory of the invariant Laplacian on the disk and the sphere -- a complex analysis approach}, in preparation.
  
  
  
\bibitem{Hub} J.H.~Hubbard, \emph{Teichm\"uller theory and applications to geometry, topology, and dynamics}. Vol. 1. 
 Matrix Editions, Ithaca, NY, 2006. 
\bibitem{6} S. Lamy, Sur la structure du groupe d'automorphismes de certaines surfaces affines [On the structure of the automorphism group of
certain non compact surfaces] \emph{Publ. Mat.} 49 (2005), no. 1, 3--20. 
\bibitem{7} F. Forstneri\v{c}, \emph{Stein manifolds and holomorphic mappings. The homotopy principle in complex analysis}, Second edition, 
Springer, 2017.

\bibitem{KK} S. Kaliman, F. Kutzschebauch, Density property for hypersurfaces $UV=P(\overline{X})$, \emph{Math. Z.} 258 (2008), no. 1, 115--131. 

\bibitem{Katok} S. Katok, \emph{Fuchsian Groups}, Univ. of Chicago Press, 1992.
	
\bibitem{KS07diff}
		S.-A.~Kim and T.~Sugawa, Invariant differential operators associated with a
                  conformal metric, \emph{Michigan Math. J}.~\textbf{55} (2007), 459--479.
                
\bibitem{8} K. Kodaira, Complex structures on $S^1\times S^3$, \emph{Proc. Natl. Acad. Sci. U S A}, 55 (1966), 240--243. 
\bibitem{9} K. Kodaira, \emph{Complex Manifolds and Deformation of Complex Structures}, 
  Springer, 2005.

\bibitem{KRSW}   D.~Kraus, O.~Roth, M.~Sch\"otz, S.~Waldmann,
  A convergent star product on the Poincar\'e disc, \emph{J.~Funct.~Anal.} 277 (2019), 2734--2771.
  
\bibitem{McKay} B. McKay, Complex homogeneous surfaces, \emph{J. Lie Theory} 25 (2015), 579--612.  

\bibitem{Annika} A.~Moucha, \textit{Spectral synthesis of the invariant Laplacian and complexified spherical harmonics}, in preparation.

  
\bibitem{RJ} A. Rabeie, Z. Jalilian, Comment on ``Coherent states on spheres'' [J. Math. Phys. 43, 1211 (2002)], 
\emph{Journal of Mathematical Physics} 52 (2011), 084101--084101-2.  

\bibitem{Range} R.~M.~Range, \emph{Holomorphic Functions and Integral Representations in Several Complex Variables}, Springer 1986.

\bibitem{SchmittSchoetz2022}
		P.~Schmitt, M.~Sch\"otz, Wick rotations in deformation quantization, \textit{Reviews in Mathematical Physics} \textbf{34} No.~1 (2022) 2150035.

\bibitem{Sz} R. Sz\Hungarian{o}ke, Complex structures on tangent bundles of Riemannian manifolds, \emph{Math. Ann.} 291 (1991), no. 3, 409--428.   
  
\bibitem{To} B. Totaro, Complexifications of nonnegatively curved manifolds, \emph{J. Eur. Math. Soc.} 5 (2003), no. 1, 69--94.   
  
\bibitem{Va} D. Varolin, The density property for complex manifolds and geometric structures. II. \emph{Internat. J. Math.} 11 (2000), no. 6, 837--847.   
  
\bibitem{Th} W. Thurston, \emph{Three--Dimensional Geometry and Topology}, Volume 1, Princeton University Press, 1997.

\end{thebibliography}
\end{document}